\documentclass[amsfonts,12pt]{amsart}
\usepackage{epsfig}
\usepackage{mathtools}
\usepackage[subnum]{cases}
\usepackage{amsmath,amssymb}
\newtheorem{theorem}{Theorem}[section]
\newtheorem{definition}[theorem]{Definition}
\newtheorem{lemma}[theorem]{Lemma}
\newtheorem{corollary}[theorem]{Corollary}
\newtheorem{remark}[theorem]{Remark}
\newtheorem{proposition}[theorem]{Proposition}

\newtheorem{problem}[theorem]{Problem}

\newcommand{\bd}{{\mathrm{bd}}\,}
\newcommand{\inte}{{\mathrm{int}}\,}
\newcommand{\reg}{{\mathrm{reg}}\,}

\newcommand{\conv}{{\mathrm{conv}}\,}

\newcommand{\ee}{\varepsilon}

\newcommand{\R}{\mathbb{R}}

\newcommand{\N}{\mathbb{N}}

\newcommand{\cS}{\mathcal{S}}

\def\cS{\mathcal{S}}
\def\cK{\mathcal{K}}
\def\sphere{S^{n-1}}
\def\N{\mathbb{N}}
\def\Rn{{\mathbb R^n}}
 \def\R{\mathbb{R}}
\def\cH{\mathcal{H}}

\def\cV{\widetilde{C}_{\phi}}
\def\ccV{\widetilde{C}_{\phi,\psi}}
\def\deV{\widetilde{C}_{H, \psi}}

\def\leV{\widetilde{C}_{G}}
\def\deV{\widetilde{C}_{G, \psi}}
\def\Oadd{\widetilde{+}_{\varphi}}
\def\teVq{\widetilde{V}_q}

\def\HeVq{\widetilde{V}_{q, \varphi}}

\def\ball{B^n}

\def\of{\overline{\Phi}}
\def\uf{\underline{\Phi}}
\def\wp{G}
\def\dveV{\widetilde{V}_G}
\def\eV{\overline{V}_{\phi}}
\def\teV{\underline{V}_{\phi}}
\def\veV{\widetilde{V}_{\phi,\varphi}}
\def\oveV{\breve{V}_{\phi,\varphi}}
\def\ovepV{\breve{V}_{\phi,\varphi,\psi}}
\def\veVt{\widetilde{V}_{\phi,\varphi_2}}

\def\bt{\begin{theorem}}
\def\et{\end{theorem}}
\def\bl{\begin{lemma}}
\def\el{\end{lemma}}
\def\br{\begin{remark}}
\def\er{\end{remark}}
\def\bc{\begin{corollary}}
\def\ec{\end{corollary}}
\def\bd{\begin{definition}}
\def\ed{\end{definition}}
\def\bp{\begin{proposition}}
\def\ep{\end{proposition}}

\begin{document}

\title{GENERAL VOLUMES IN THE ORLICZ-BRUNN-MINKOWSKI THEORY AND A RELATED MINKOWSKI PROBLEM I}
\author[Richard J. Gardner, Daniel Hug, Wolfgang Weil, Sudan Xing, and Deping Ye]
{Richard J. Gardner, Daniel Hug, Wolfgang Weil, Sudan Xing, and Deping Ye}
\address{Department of Mathematics, Western Washington University,
Bellingham, WA 98225-9063, USA} \email{richard.gardner@wwu.edu}
\address{Karlsruhe Institute of Technology, Department of Mathematics,
D-76128 Karlsruhe, Germany}
\email{daniel.hug@kit.edu}
\address{Karlsruhe Institute of Technology, Department of Mathematics,
D-76128 Karlsruhe, Germany} \email{wolfgang.weil@kit.edu}
\address{Department of Mathematics and Statistics, Memorial University of Newfoundland,\newline St.~John's, Newfoundland, Canada A1C 5S7} \email{sudanxing@gmail.com}
\address{Department of Mathematics and Statistics, Memorial University of Newfoundland,\newline St.~John's, Newfoundland, Canada A1C 5S7} \email{deping.ye@mun.ca}
\thanks{First author supported in
part by U.S.~National Science Foundation Grant DMS-1402929 and the Alexander von Humboldt Foundation.   Second and third authors supported in part by German Research Foundation (DFG) grants HU 1874/4-2, WE 1613/2-2, and FOR 1548. Fifth author supported in part by an NSERC grant.}
\subjclass[2010]{Primary: 52A20, 52A30; secondary: 52A39, 52A40} \keywords{curvature measure, dual curvature measure, Minkowski problem, Orlicz addition, Orlicz-Brunn-Minkowski theory.}
\maketitle
\pagestyle{myheadings}
\markboth{RICHARD J. GARDNER, DANIEL HUG,
WOLFGANG WEIL, SUDAN XING, AND DEPING YE}{THE ORLICZ-BRUNN-MINKOWSKI THEORY AND A MINKOWSKI PROBLEM}

\begin{abstract}
The general volume of a star body, a notion that includes the usual volume, the $q$th dual volumes, and many previous types of dual mixed volumes, is introduced.  A corresponding new general dual Orlicz curvature measure is defined that specializes to the $(p,q)$-dual curvature measures introduced recently by Lutwak, Yang, and Zhang.  General variational formulas are established for the general volume of two types of Orlicz linear combinations.  One of these is applied to the Minkowski problem for the new general dual Orlicz curvature measure, giving in particular a solution to the Minkowski problem posed by Lutwak, Yang, and Zhang for the $(p,q)$-dual curvature measures when $p>0$ and $q<0$. A dual Orlicz-Brunn-Minkowski inequality for general volumes is obtained, as well as dual Orlicz-Minkowski-type inequalities and uniqueness results for star bodies.  Finally, a very general Minkowski-type inequality, involving two Orlicz functions, two convex bodies, and a star body, is proved, that includes as special cases several others in the literature, in particular one due to Lutwak, Yang, and Zhang for the $(p,q)$-mixed volume.
\end{abstract}

\section{Introduction}\label{intro}
The classical Brunn-Minkowski theory was developed by Minkowski, Aleksandrov, and many others into the powerful tool it is today.  It focuses on compact convex sets and their orthogonal projections and metric properties such as volume and surface area, but has numerous applications beyond geometry, both within and outside mathematics.  In recent decades it has been significantly extended in various ways.  Germinating a seed planted by Firey, Lutwak \cite{L3} brought the $L_p$-Brunn-Minkowski theory to fruition.  A second extension, the Orlicz-Brunn-Minkowski theory, arose from work of Ludwig \cite{Lud}, Ludwig and Reitzner \cite{LudR}, and Lutwak, Yang, and Zhang \cite{LYZ7, LYZ8}. Each theory has a dual counterpart treating star-shaped sets and their intersections with subspaces, and these also stem from the pioneering work \cite{L1} of Lutwak.  The main ingredients in each theory are a distinguished class of sets, a notion of volume, and an operation, usually called addition, that combines two or more sets in the class.  Each theory has been described and motivated at length in previous work, so we refer the reader to Schneider's classic treatise \cite{Sch} and the introductions of the articles \cite{GHW,GHW2014,ghwy15}, and will focus henceforth on the contributions made in the present paper.

Our work is inspired by the recent groundbreaking work of Huang, Lutwak, Yang, and Zhang \cite{LYZActa} and Lutwak, Yang, and Zhang \cite{LYZ-Lp}.  In \cite{LYZActa}, the various known measures that play an important part in the Brunn-Minkowski theory---the classical area and curvature measures and their $L_p$ counterparts---were joined by new dual curvature measures, and surprising relations between them were discovered, revealing fresh connections between the classical and dual Brunn-Minkowski theories.  These connections were reinforced in the sequel \cite{LYZ-Lp}, which defined the very general $L_p$ dual curvature measures that involve both convex and star bodies and {\em two} real parameters $p$ and $q$. With each measure comes the challenge of solving the corresponding Minkowski problem, a fundamental endeavor that goes back to the original work of Minkowski and Aleksandrov.

The present paper focuses on the Orlicz-Brunn-Minkowski theory.  Just as Orlicz spaces generalize $L_p$ spaces, the Orlicz theory brings more generality, but presents additional challenges due to the loss of homogeneity.  Here we introduce very general dual Orlicz curvature measures which specialize to both the $L_p$ dual curvature measures in \cite{LYZ-Lp} and the dual Orlicz curvature measures defined in \cite{XY2017-1, ZSY2017}.  We state the corresponding Minkowski problem and present a partial solution, though one general enough to include those from \cite{XY2017-1, ZSY2017} as well as solving the case $p>0$ and $q<0$ of the Minkowski problem posed in \cite[Problem~8.1]{LYZ-Lp}.  (After we proved our result, we learned that B\"{o}r\"{o}czky and Fodor \cite{BorFod} have solved the case $p>1$ and $q>0$.  The authors of \cite{BorFod} state that
Huang and Zhao have also solved the case $p>0$ and $q<0$ in the unpublicized manuscript \cite{HuangZhao}.) The Minkowski problem in \cite[Problem~8.1]{LYZ-Lp} requires finding, for given $p,q\in \R$, $n$-dimensional Banach norm $\|\cdot\|$, and $f:\sphere\to [0,\infty)$, an $h:\sphere\to (0,\infty)$ that solves the Monge-Amp\`{e}re equation
\begin{equation}\label{LYZPDE}
h^{1-p}\,\|\bar{\nabla}h+h\iota \|^{q-n}\det(\bar{\nabla}^2h+hI)=f
\end{equation}
on the unit sphere $\sphere$, where $\bar{\nabla}$ and $\bar{\nabla}^2$ are the gradient vector and Hessian matrix of $h$, respectively, with respect to an orthonormal frame on $\sphere$, $\iota$ is the identity map on $\sphere$, and $I$ is the identity matrix.  The equation (\ref{LYZPDE}) is derived in \cite[(5.8), p.~116]{LYZ-Lp}; previous Minkowski problems correspond to taking $p=0$ and $\|\cdot\|=|\cdot|$, the Euclidean norm (the dual Minkowski problem from \cite{LYZActa}), $q=0$ and $\|\cdot\|=|\cdot|$ (the $L_p$ Aleksandrov problem), and $q=n$ (the $L_p$ Minkowski problem, which reduces to the classical Minkowski problem when $p=1$).

We refer the reader to the introductions of \cite{LYZActa, LYZ-Lp} and to \cite[Sections~8.2 and 9.2]{Sch} for detailed discussions and references to the extensive literature on these problems.

Also introduced here are new generalizations of volume.  Let $G:(0,\infty)\times S^{n-1}\to (0,\infty)$ be continuous (see Section~\ref{prel} for definitions and notation).  The general dual volume $\dveV(K)$ of a star body $K$ is defined by
$$\dveV(K)=\int_{S^{n-1}}G(\rho_K(u),u)\,du,$$
where $\rho_K$ is the radial function of $K$, giving the distance from the origin to the boundary of $K$ in the direction $u$, while the general volume of a convex body $K$ is defined by
$$V_G(K)=\int_{S^{n-1}}G(h_K(u),u)\,dS(K,u),$$
where $h_K$ is the support function and $S(K,\cdot)$ is the surface area measure of $K$.  (Integrals with respect to the $i$th area measures $S_i(K,\cdot)$, $1\le i\le n-1$, may also be considered.)  The novel feature here is the extra argument $u$ in $G$; this allows $\dveV(K)$ and $V_G(K)$ to include not only the usual volume and variants of it, but also many of the mixed and dual mixed volumes that have previously been found useful in the literature.  The same function $G(t,u)$ is behind our general dual Orlicz curvature measures (see Definition~\ref{general dual Orlicz curvature measure-11-27}).  The present paper focuses mainly on the dual theory, so from the outset we work with the general dual volume $\dveV(K)$ and obtain variational formulas (necessary for the Minkowski problem) for it. The corresponding study for $V_G(K)$ and the classical theory is to be carried out in \cite{ghwxy}. It should be mentioned that in this context, Orlicz-Minkowski problems were first investigated by Haberl, Lutwak, Yang, and Zhang \cite{HLYZ}.

The general dual Orlicz curvature measures mentioned above arise naturally from the general dual volumes and are denoted by $\deV(K,\cdot)$, where $G:(0,\infty)\times S^{n-1}\to (0,\infty)$ and $\psi:(0,\infty)\to(0,\infty)$ are continuous.  The corresponding Minkowski problem is:

{\em For which nonzero finite Borel measures $\mu$ on $\sphere$ and continuous functions $\wp$ and $\psi$ do there exist $\tau\in \R$ and $K\in \cK_{o}^n$ such that $\mu=\tau\,\deV(K,\cdot)$?}

In our partial solution, presented in Theorem~\ref{solution-general-dual-Orlicz-main theorem-11-27} below, the lack of homogeneity necessitates extra care in the variational method we employ.  The problem requires finding, for given $G$, $\psi$, and $f:\sphere\to [0,\infty)$, an $h:\sphere\to (0,\infty)$ and $\tau\in \R$ that solve the Monge-Amp\`{e}re equation
\begin{equation}\label{new}
\frac{\tau h}{\psi\circ h}\,P(\bar{\nabla}h+h\iota)
\,\det(\bar{\nabla}^2h+hI)=f,
\end{equation}
where $P(x)=|x|^{1-n}G_t(|x|,\bar{x})$. Equation (\ref{new}) is derived before Theorem~\ref{solution-general-dual-Orlicz-main theorem-11-27} in a brief discussion where we also show that (\ref{new}) is more general than (\ref{LYZPDE}).

In a third contribution, we prove very general Orlicz inequalities of the Minkowski and Brunn-Minkowski type which include others in the literature, such as \cite[Theorem~7.4]{LYZ-Lp}, as special cases.  Some general uniqueness theorems are also demonstrated.

The paper is organized as follows.  The preliminary Section~\ref{prel} gives definitions and notation, as well as the necessary background on two types of Orlicz linear combination.  In Section~\ref{section 2.1}, we define the new general dual volumes and general dual Orlicz curvature measures.  Sections~\ref{section4} and~\ref{section5} contain our variational formulas.  In Section~\ref{section-6}, we state our Minkowski problem and provide a partial solution (see Problem~\ref{Minkowski-c-11-28} and Theorem~\ref{solution-general-dual-Orlicz-main theorem-11-27}).  Dual Orlicz-Brunn-Minkowski inequalities can be found in Section~\ref{section7} and dual Orlicz-Minkowski inequalities and uniqueness results are the focus of Section~\ref{section8}.
	
\section{Preliminaries and Background}\label{prel}
We use the standard notations $o$, $\{e_1, \ldots, e_n\}$, and $\| \cdot \|$ for the origin, the canonical orthonormal basis, and a norm, respectively, in $\Rn$. The Euclidean norm and inner product on $\Rn$ are denoted by $|\cdot|$ and $\langle \cdot, \cdot \rangle$, respectively. Let $\ball=\{x\in \Rn: |x|\leq 1\}$ and $\sphere=\{x\in \Rn: |x|=1\}$ be the unit ball and sphere in $\Rn$.  The characteristic function of a set $E$ is signified by $1_E$.

We write ${\mathcal{H}}^k$ for $k$-dimensional Hausdorff measure in $\R^n$, where $k\in \{1,\dots,n\}$.  For compact sets $E$, we also write $V_n(E)={\mathcal{H}}^n(E)$ for the {\em volume} of $E$.  The notation $dx$ means $d{\mathcal{H}}^k(x)$ for the appropriate $k=1,\dots,n$, unless stated otherwise. In particular, integration on $\sphere$ is usually denoted by $du=d{\mathcal{H}}^{n-1}(u)$.

The class of nonempty compact convex sets in $\R^n$ is written $\cK^n$. We will often work with $\cK_o^n$, the set of {\em convex bodies} (i.e., compact convex subsets in $\Rn$ with nonempty interiors) {\em containing $o$ in their interiors}.  For the following information about convex sets, we refer the reader to \cite{Gruber2007, Sch}. The standard metric on $\cK^n$ is the {\em Hausdorff metric} $\delta(\cdot, \cdot)$, which can be defined by
$$\delta(K, L)=\|h_K-h_ {L}\|_{\infty} =\sup_{u\in \sphere} |h_K(u)-h_{L}(u)|$$
for $K, L\in \cK^n$, where $h_K: \sphere\rightarrow \R$ is the {\em support function} of $K\in \cK^n$, given by $h_K(u)=\sup_{x\in K}\langle u, x\rangle$ for $u\in \sphere$. We say that the sequence $K_1, K_2,\ldots$ of sets in $\cK^n$ {\em converges} to $K\in \cK^n$ if and only if $\lim_{i\rightarrow \infty} \delta(K_i, K)=0$.  The {\em Blaschke selection theorem} states that every bounded sequence in $\cK^n$ has a subsequence that converges to a set in $\cK^n$. The {\em surface area measure} $S(K,\cdot)$ of a convex body $K$ in $\R^n$ is defined for Borel sets $E\subset S^{n-1}$ by
\begin{equation}\label{surface:area:1}
S(K,E)=\mathcal{H}^{n-1}(\nu_{K}^{-1}(E)),
\end{equation}
where $\nu_{K}^{-1}(E)=\{x\in \partial K: \  \nu_K(x)\in E\}$ is the {\em inverse Gauss map} of $K$ (see Section~\ref{Maps}).

Let $\mu$ be a nonzero finite Borel measure on $\sphere$.  We say that $\mu$ is {\em not concentrated on any closed hemisphere}  if \begin{equation}\label{condition for Minkowski problem}
\int _{\sphere} \langle u, v\rangle_+ \,d\mu(u)>0\ \ \ \mathrm{for~} v\in \sphere,
\end{equation}
where  $a_+=\max\{a, 0\}$ for $a\in \R$. We write $|\mu|=\mu(S^{n-1})$.

As usual, $C(E)$ denotes the class of continuous functions on $E$ and we shall write $C^+(E)$ for the {\em strictly} positive functions in $C(E)$. Let $\Omega\subset  S^{n-1}$ be a closed set not contained in any closed hemisphere of $S^{n-1}$.  For each   $f\in C^+(\Omega)$, one can define a convex body $[f]$, the {\em Aleksandrov body} (or {\em Wulff shape}), associated to it, by setting
$$
[f]= \bigcap_{u\in\Omega}\big\{x\in\R^n:\langle x, u\rangle\leq f(u)\big\}.
$$
In particular, when $\Omega=\sphere$ and $f=h_K$ for $K\in \cK^n$, one has
$$K=[h_K]=\bigcap_{u\in\sphere}
\big\{x\in\R^n:\langle x, u\rangle\leq h_K(u)\big\}.$$
Note that
$$H(K,u)=\big\{x\in\R^n: \langle x, u\rangle=h_K(u)\big\}$$
is the supporting hyperplane of $K$ in the direction $u\in \sphere$.

A set $L$ in $\R^n$ is {\em star-shaped about $o$} if $o\in L$ and $\lambda x\in L$ for $x\in L$ and $\lambda\in [0, 1]$. For each such $L$ and for $x\in \Rn\setminus\{o\}$, let
$$\rho_L(x)=\sup\{\lambda>0: \lambda x\in L\}.$$
Then $\rho_L: \Rn\setminus \{o\}\rightarrow \R$ is called the {\em radial function} of $L$. The function $\rho_L$ is homogeneous of degree $-1$, that is, $\rho_L(rx)=r^{-1}\rho_L(x)$ for $x\in \R^n\setminus \{o\}$.  This allows us to consider $\rho_L$ as a function on $\sphere$.  Let $\cS^n$ be the class of star-shaped sets in $\R^n$  about $o$ whose radial functions are bounded Borel measurable functions on $\sphere$.  The class of $L\in \cS^n$ with $\rho_L>0$ is denoted by $\cS_+^n$ and the class $\cS_{c+}^n$ of {\em star bodies} comprises those $L\in \cS_+^n$ such that $\rho_L$ is continuous on $\sphere$.  If $L\in \cS_+^n$, then $\rho_L(u)u\in \partial L$ and $\rho_L(x)=1$ for $x\in \partial L$, the boundary of $L$.   The natural metric on $\cS^n$ is the {\em radial metric} $\widetilde{\delta}(\cdot, \cdot)$, which can be defined by
$$\widetilde{\delta}(L_1, L_2)=\|\rho_{L_1}-\rho_{L_2}\|_{\infty}=\sup_{u\in \sphere} |\rho_{L_1}(u)-\rho_{L_2}(u)|,$$
for $L_1, L_2\in \cS^n$.  Consequently, we can define convergence in $\cS^n$ by $\lim_{j\rightarrow \infty} \widetilde{\delta}(L_j, L)=0$ for $L,L_1, L_2,\ldots\in \cS^n$.  Clearly, $\cK_o^n\subset \cS_{c+}^n$.  It follows directly from the relations between the metrics $\delta$ and $\widetilde{\delta}$ in \cite[Lemma~2.3.2, (2.3.15) and (2.3.16)]{Gro} that if $K, K_1, K_2,\ldots \in \cK_o^n$, then $K_i\rightarrow K$ in the Hausdorff metric if and only if $K_i\rightarrow K$ in the radial metric.

If $K\in\mathcal{K}_{o}^n$, the {\em polar body} $K^*$ of $K$ is defined by
$$K^*=\{x\in\mathbb{R}^n: \langle x, y\rangle\leq1 \ \text{for}\  y\in K \}.$$
Then $(K^*)^*=K$ and (see \cite[(1.52), p.~57]{Sch})
\begin{eqnarray} \label{bi-polar--1}
\rho_K(x) h_{K^*}(x)=h_K (x) \rho_{K^*}(x)=1\ \ \ \mathrm{for~} x\in \R^n\setminus\{o\}.
\end{eqnarray}

One can define convex bodies associated to radial functions of star bodies. In general, if $\Omega\subset S^{n-1}$ is a closed set not contained in any closed hemisphere of $S^{n-1}$, and $f\in C^+(\Omega)$, define $\langle f\rangle\in\mathcal{K}_{o}^n$, the {\em convex hull} of $f$, by
\begin{eqnarray*}
\langle f\rangle =\conv\{f(u)u: u\in\Omega\}.
\end{eqnarray*}  The properties of $\langle f\rangle$ are similar to those of the Aleksandrov body. In particular, taking $\Omega=S^{n-1}$, we have $\langle \rho_K \rangle=K$ for each $K\in \cK_o^n$. It can be checked (see  \cite[Lemma 2.8]{LYZActa}) that
\begin{equation}\label{relation}
[f]^*=\langle 1/f \rangle.
\end{equation}

Throughout the paper, we will need certain classes of functions $\varphi: (0, \infty)\rightarrow (0, \infty)$. Let
\begin{eqnarray*}
\mathcal{I}  &=& \{\varphi  \ \mbox{is continuous and strictly increasing with}\ \varphi(1)=1, \varphi(0)=0, \ \mbox{and}\ \varphi(\infty)=\infty\}, \\ \mathcal{D}  &=& \{\varphi   \ \mbox{is continuous and strictly decreasing with} \ \varphi(1)=1, \varphi(0)=\infty, \ \mbox{and}\ \varphi(\infty)=0\},
\end{eqnarray*}
where $\varphi(0)$ and $\varphi(\infty)$ are considered as limits, $\varphi(0)=\lim_{t\rightarrow 0^+} \varphi(t)$ and $\varphi(\infty)=\lim_{t\rightarrow \infty} \varphi(t)$. Note that the values of $\varphi$ at $t=0, 1, \infty$ are chosen for technical reasons; results may still hold for other values of $\varphi$ at  $t=0, 1, \infty$.

For $a\in \R\cup\{-\infty\}$, we also require the following class of functions $\varphi:(0,\infty)\to (a,\infty)$:
$${\mathcal{J}}_a=\{\mbox{$\varphi$ is continuous and strictly monotonic,   $\inf_{t>0}\varphi(t)=a$, and  $\sup_{t>0}\varphi(t)=\infty$}\}.$$
Note that the log function belongs to ${\mathcal{J}}_{-\infty}$ and $\mathcal{I}\cup \mathcal{D}\subset \mathcal{J}_0$.

Let $f_0\in C^+(\sphere)$, let $g\in C(\sphere)$, and let $\varphi\in {\mathcal{J}}_a$ for some $a\in \R\cup\{-\infty\}$. Then $\varphi^{-1}: (a,\infty)\to (0,\infty)$, and since $S^{n-1}$ is compact, we have $0<c\le f_0\le C$ for some $0<c\le C$.  It is then easy to check that for $\ee\in \R$ close to $0$, one can define $f_{\ee}=f_{\ee}(f_0,g,\varphi)\in C^+(\sphere)$ by
\begin{eqnarray}\label{genplus} f_{\ee}(u)= {\varphi}^{-1}\left(\varphi(f_0(u))+\ee g(u)\right).
\end{eqnarray}
Note that we can apply (\ref{genplus}) when $f_0=h_K$ for some $K\in \cK^n_o$ or when $f_0=\rho_K$ for some $K\in \cS^n_{c+}$. Sometimes we will use this definition when $S^{n-1}$ is replaced by a closed set $\Omega\subset  S^{n-1}$ not contained in any closed hemisphere of $S^{n-1}$.

The {\em left derivative} and {\em right derivative} of a real-valued function $f$ are denoted by $f'_l$ and $f'_r$, respectively.  Whenever we use this notation, we assume that the one-sided derivative exists.

\subsection{Orlicz linear combination}
Let $K, L\in \cK_o^n$. For $\ee>0$, and either $\varphi_1,\varphi_2\in \mathcal{I}$ or $\varphi_1,\varphi_2\in \mathcal{D}$, define $h_{\ee}\in C^+(\sphere)$ (implicitly and uniquely) by
\begin{equation}\label{Orlicz-addition--1-1}
\varphi_1\left(\frac{h_K(u)}{h_{\ee}(u)}\right)+\ee\varphi_2
\left(\frac{h_L(u)}{h_{\ee}(u)}\right)=1 \ \ \ \ \ \mbox{for~}u\in \sphere.
\end{equation}
Note that $h_{\ee}=h_{\ee}(K,L,\varphi_1, \varphi_2)$ may not be a support function of a convex body unless $\varphi_1, \varphi_2\in\mathcal{I}$ are convex, in which case $h_{\ee}=h_{K+_{\varphi,\ee} L}$, where $K+_{\varphi,\ee} L$ is an {\em Orlicz linear combination of $K$ and $L$} (see \cite[p.~463]{GHW2014}).  However, the Aleksandrov body $[h_{\ee}]$ of $h_{\ee}$ belongs to $\cK_o^n$.

An alternative approach to forming Orlicz linear combinations is as follows.  Let $K\in \cK_o^n$, let $g\in C(\sphere)$, let $\varphi\in {\mathcal{J}}_a$ for some $a\in\R\cup\{-\infty\}$, and let $\widehat{h}_{\ee}$ be defined by (\ref{genplus}) with $f_0=h_K$.  This approach goes back to Aleksandrov \cite{Aleks} in the case when $\varphi(t)=t$. Again, the Aleksandrov body $[\widehat{h}_{\ee}]$ of $\widehat{h}_{\ee}$ belongs to $\cK_o^n$.  When $g=\varphi\circ h_L$ and $\varphi\in\mathcal{I}\subset {\mathcal{J}}_0$ is convex, $[\widehat{h}_{\ee}]= K\widehat{+}_{\varphi}\,\ee\cdot L$, as defined in \cite[(10.4), p.~471]{GHW2014}.

Suppose that $K, L\in \cK_o^n$, that  $\varphi\in\mathcal{I}$ is convex, and that $K+_{\varphi,\ee} L$ is defined by (\ref{Orlicz-addition--1-1}) with $\varphi_1=\varphi_2=\varphi$.  Then both $K+_{\varphi,\ee} L$ and $K\widehat{+}_{\varphi}\,\ee\cdot L$ belong to $\cK_o^n$ and coincide when $\varphi(t)=t^p$ for some $p\ge 1$, but they differ in general (to see this, compare the corresponding different variational formulas given by \cite[(8.11) and (8.12), p.~466]{GHW2014} and \cite[p.~471]{GHW2014}).

It is known (see \cite[Lemma~8.2]{GHW2014}, \cite[p.~18]{HongYeZhang-2017}, and \cite[Lemma~3.2]{XJL}) that $h_{\ee}\rightarrow h_K$ and $\widehat{h}_{\ee}\rightarrow h_K$ uniformly on $\sphere$ as $\ee \rightarrow 0$ and hence, by \cite[Lemma~7.5.2]{Sch}, both $[h_{\ee}]$ and $[\widehat{h}_{\ee}]$ converge to $K\in\cK_o^n$ as $\ee \rightarrow 0$.  Part (ii) of the following lemma is proved in \cite[(5.38)]{HongYeZhang-2017} for the case when $\varphi\in{\mathcal{I}}\cup{\mathcal{D}}$, but the same proof applies to the more general result stated.

\bl \label{unif-conv-two-functions} Let $K, L\in\mathcal{K}_o^n$.

\noindent
{\rm{(i)}} {\rm (\cite[Lemma~8.4]{GHW2014}, \cite[Lemma~5.2]{XJL}.)} If $\varphi_1,\varphi_2\in \mathcal{I}$ and $(\varphi_1)'_l(1)>0$, then
\begin{eqnarray}
\lim_{\ee\rightarrow
0^+}\frac{h_\ee(u)-h_K(u)}{\ee}
&=&\frac{h_K(u)}{(\varphi_1)'_l(1)}\,\varphi_2\!\left(\frac{h_L(u)}
{h_K(u)}\right) \label{convergence of f-e}
\end{eqnarray}
uniformly on $\sphere$.  For $\varphi_1,\varphi_2\in \mathcal{D}$,  \eqref{convergence of f-e} holds when $(\varphi_1)'_r(1)<0$, with $(\varphi_1)'_l(1)$ replaced by $(\varphi_1)'_r(1)$.

\smallskip

\noindent
{\rm{(ii)}} {\rm (cf.~\cite[(5.38)]{HongYeZhang-2017}.)} Let $a\in \R\cup\{-\infty\}$. If $\varphi\in {\mathcal{J}}_a$ and $\varphi'$ is continuous and nonzero on $(0, \infty)$, then for $g\in C(\sphere)$,
$$
\lim_{\ee\rightarrow 0} \frac{\widehat{h}_{\ee}(u)-h_{K}(u)}{\ee}= \frac{g(u)}{\varphi'\left(h_{K}(u)\right)}
$$
uniformly on $\sphere$, where $\widehat{h}_{\ee}$ is defined by \eqref{genplus} with $f_0=h_K$.
\el

Analogous results hold for radial functions of star bodies.
Let $K, L\in \cS_{c+}^n$. For $\ee>0$, and  either $\varphi_1,\varphi_2\in \mathcal{I}$ or $\varphi_1,\varphi_2\in \mathcal{D}$,  define $\rho_{\ee}\in C^+(\sphere)$ (implicitly and uniquely) by
\begin{equation}\label{rlc} \varphi_1\left(\frac{\rho_K(u)}{\rho_{\ee}(u)}\right)
+\ee\varphi_2\left(\frac{\rho_L(u)}{\rho_{\ee}(u)}\right)=1 \ \ \ \ \ \mbox{for~}u\in \sphere.
\end{equation}
Then $\rho_{\ee}$ is the radial function of the {\em radial Orlicz linear combination $K\widetilde{+}_{\varphi,\ee}L$ of $K$ and $L$} (see \cite[(22), p.~822]{ghwy15}).

Let $a\in \R\cup\{-\infty\}$. For $\varphi\in {\mathcal{J}}_a$,  $g\in C(\sphere)$, and $\ee\in \R$ close to $0$, define $\widehat{\rho}_{\ee}\in C^+(\sphere)$ by (\ref{genplus}) with $f_0=\rho_K$. The definitions of both $\rho_{\ee}$ and $\widehat{\rho}_{\ee}$ can be extended to $K, L\in \cS_+^n$ (or even $L\in \cS^n$), but we shall mainly work with star bodies and hence focus on $\cS_{c+}^n$.  It is known (see \cite[Lemma~5.1]{ghwy15}, \cite[p.~18]{HongYeZhang-2017} (with $h$ replaced by $\rho$), and \cite[Lemma~3.5]{Zhub2014}) that $\rho_{\ee}\rightarrow \rho_K$ and $\widehat{\rho}_{\ee}\rightarrow \rho_K$ uniformly on $\sphere$ as $\ee \rightarrow 0$.  From this and the equivalence between convergence in the Hausdorff and radial metrics for sets in $\cK_o^n$, one sees that, for each $K\in \cK_o^n$, both $\langle \rho_{\ee}\rangle$ and $\langle \widehat{\rho}_{\ee}\rangle$ converge to $K$ in either metric.

\bl \label{unif-conv-two-functions-2} Let $K, L\in\mathcal{S}_{c+}^n$.

\noindent
{\rm{(i)}} {\rm (\cite[Lemma~5.3]{ghwy15}; see also \cite[Lemma~4.1]{Zhub2014}.)} If $\varphi_1,\varphi_2\in \mathcal{I}$ and $(\varphi_1)'_l(1)>0$, then
\begin{eqnarray}
\lim_{\ee\rightarrow
0^+}\frac{\rho_\ee(u)-\rho_K(u)}{\ee}
&=&\frac{\rho_K(u)}{(\varphi_1)'_l(1)} \varphi_2\!\left(\frac{\rho_L(u)}{\rho_K(u)}\right)
\label{newlimits}
\end{eqnarray}
uniformly on $\sphere$.  For $\varphi_1,\varphi_2\in \mathcal{D}$,  \eqref{newlimits} holds when $(\varphi_1)'_r(1)>0$, with $(\varphi_1)'_l(1)$ replaced by $(\varphi_1)'_r(1)$.

\smallskip

\noindent
{\rm{(ii)}} {\rm (cf.~\cite[(5.38)]{HongYeZhang-2017}.)} Let $a\in \R\cup\{-\infty\}$. If $\varphi\in {\mathcal{J}}_a$ and $\varphi'$ is continuous and nonzero on $(0, \infty)$, then for $g\in C(\sphere)$,
\begin{equation}
\lim_{\ee\rightarrow 0} \frac{\widehat{\rho}_{\ee} (u)-\rho_{K}(u)}{\ee}= \frac{g(u)}{\varphi'\left(\rho_{K}(u)\right)} \label{conv-hat-sum}
\end{equation}
uniformly on $\sphere$, where $\widehat{\rho}_{\ee}$ is defined by \eqref{genplus} with $f_0=\rho_K$.
\el

\subsection{Maps related to a convex body}\label{Maps}
We recall some terminology and facts from \cite[Section~2.2]{LYZActa}. Let $K\in\cK_o^n$. Define
$$\pmb{\nu}_K(E)=\{u\in S^{n-1}: x\in H(K,u)\text{ for some }x\in E\}$$
for $E\subset \partial K$,
$$\pmb{x}_K(E)=\{x\in\partial K: x\in H(K,u)\text{ for some }u\in E\}$$
for $E\subset \sphere$, and
$$\pmb{\alpha}_K(E)=\pmb{\nu}_K(\{\rho_K(u)u\in\partial K: u\in E\})$$
for $E \subset \sphere$.  Let $\sigma_K\subset\partial K$, $\eta_K\subset S^{n-1}$, and $\omega_K\subset S^{n-1}$ be the sets where $\pmb{\nu}_K(\{x\})$, $\pmb{x}_K(\{u\})$, and $\pmb{\alpha}_K(\{u\})$, respectively, have two or more elements. Then
\begin{equation}\label{zerosets}
\cH^{n-1}(\sigma_K)=\cH^{n-1}(\eta_K)=\cH^{n-1}(\omega_K)=0.
\end{equation}
Elements of $\sphere\setminus\eta_K$ are called {\em regular normal vectors} of $K$ and $\reg K=\partial K\setminus\sigma_K$ is the set of {\em regular boundary points} of $K$. We write $\nu_K(x)$, $x_K(u)$, and $\alpha_K(u)$ instead of $\pmb{\nu}_K(\{x\})$, $\pmb{x}_K(\{u\})$, and $\pmb{\alpha}_K(\{u\})$ if $x\in \reg K$, $u\in\sphere\setminus \eta_K$, and $u\in \sphere \setminus \omega_K$, respectively.

Next, we define
$$
\pmb{\alpha}^*_K(E)=\{x/|x|: x\in \partial K \cap H(K,u)\text{ for some }u\in E\}=\{x/|x|: x\in \pmb{x}_K(E)\}
$$
for $E\subset \sphere$.  In particular, one can define a continuous map $\alpha^*_K(u)= x_K(u)/|x_K(u)|$ for $u\in S^{n-1}\setminus\eta_K$. For $E \subset S^{n-1}$, we have $\pmb{\alpha}^*_K(E)=\pmb{\alpha}_{K^*}(E)$. Moreover, for $\cH^{n-1}$-almost all $u\in S^{n-1}$,
\begin{equation}\label{map-reverse}
\alpha^*_K(u)=\alpha_{K^*}(u)
\end{equation}
and
\begin{equation} \label{relation-11-27}
u\in \pmb{\alpha}_K^*(E)\ \text{if and only if}\  \alpha_K(u)\in E.
\end{equation}

\section{General dual volumes and curvature measures}\label{section 2.1}
Let $\wp:(0, \infty)\times \sphere\to (0,\infty)$ be continuous.  (Remark~\ref{remfeb25} addresses the possibility of allowing $\wp:(0, \infty)\times \sphere\to \R$.)  For $K\in \cS^n_+$, define the {\em general dual volume} $\dveV(K)$ of $K$ by
\begin{equation}\label{def-H-volume-12-07}
\dveV(K)=\int_{\sphere} \wp(\rho_K(u), u)\,du.
\end{equation}
Our approach will be to obtain results for this rather general set function that yield geometrically interesting consequences for particular functions $\wp$.

Let $\phi:\mathbb{R}^n\setminus\{o\}\rightarrow(0, \infty)$ be a continuous function.  One special case of interest is when $\wp=\of$, where
\begin{equation}\label{Phidef}
\of(t, u)=\int_{t}^{\infty}\phi(ru)r^{n-1}\,dr
\end{equation}
for $t>0$ and $u\in \sphere$. Then we define $\eV(K)=\widetilde{V}_{\of}(K)$, so that
\begin{eqnarray}\label{general-quermass-11}
\eV(K)  = \int_{\sphere}\of(\rho_{K}(u),u)\,du=
\int_{\sphere}\int_{\rho_K(u)}^{\infty}\phi(ru)r^{n-1}\,dr\,du=
\int_{\mathbb{R}^n \setminus K} \phi(x)\,dx,
\end{eqnarray}
where the integral may be infinite.  Similarly, taking $\wp=\uf$, where
$$
\uf(t, u)=\int^{t}_0\phi(ru)r^{n-1}\,dr
$$
for $t>0$ and $u\in \sphere$, we define $\teV(K)=\widetilde{V}_{\uf}(K)$, whence
\begin{eqnarray}\label{general-quermass-22}
\teV(K) =\int_{\sphere}\uf(\rho_{K}(u),u)\,du= \int_{K}\phi(x)\,dx,
\end{eqnarray}
where again the integral may be infinite.  We refer to both $\teV(K)$ and $\eV(K)$ as a {\em general dual Orlicz volume} of $K\in \cS^n$.  Indeed, if $q\neq 0$ and $\phi(x)=(|q|/n)|x|^{q-n}$, then
$$
\widetilde{V}_{q}(K)=\frac{1}{n}\int _{\sphere}\rho_K(u)^q\,du=\begin{cases}
\eV(K),& {\text{if $q<0$,}}\\
\teV(K),& {\text{if $q>0$}},
\end{cases}
$$
is the {\em $q$th dual volume of $K$}; see \cite[p.~410]{Gar06}. In particular, when $q=n$, we have $\teV(K)=V_n(K),$ the volume of $K$. More generally, if $\phi(x)=(|q|/n)|x|^{q-n}\rho_Q(x/|x|)^{n-q}$, where $q\neq 0$ and $Q\in \cS^n$, then
\begin{equation}\label{qmixedv}
\widetilde{V}_{q}(K,Q)=\frac{1}{n}\int _{\sphere}\rho_K(u)^q\rho_Q(u)^{n-q}\,du=\begin{cases}
\eV(K),& {\text{if $q<0$,}}\\
\teV(K),& {\text{if $q>0$}},
\end{cases}
\end{equation}
is the {\em $q$th dual mixed volume of $K$ and $Q$}; see \cite[p.~410]{Gar06}.

Other special cases of $\dveV(K)$ of interest, the general Orlicz dual mixed volumes $\veV(K, L)$ and $\oveV(K,g)$, are given in (\ref{dua1}) and (\ref{dua2}).

Next, we introduce a new general dual Orlicz curvature measure.

\bd\label{general dual Orlicz curvature measure-11-27}
{\rm Let $K\in \cK_o^n$, let $\psi: (0, \infty)\rightarrow (0, \infty)$ be continuous, and let $\wp_t(t,u)=\partial \wp(t,u)/\partial t$ be such that $u\mapsto\wp_t(\rho_K(u),u)$ is integrable on $S^{n-1}$.  Define the finite {\em signed} Borel measure $\deV(K, \cdot)$ on $S^{n-1}$ by
\begin{equation}\label{gencdef}
\deV(K, E)=\frac{1}{n}\int_{\pmb{\alpha}^*_K(E)} \frac{\rho_{K}(u)\, \wp_t(\rho_K(u), u) }{\psi(h_{K}(\alpha_K(u)))}\,du
\end{equation}
for each Borel set $E\subset\sphere$.  If $\psi\equiv1$, we often write $\leV(K, \cdot)$ instead of $\deV(K, \cdot)$.}
\ed

To see that $\deV(K,\cdot)$ is indeed a finite signed Borel measure on $S^{n-1}$, note firstly that $\deV(K,\emptyset)=0$.  Since $K\in \cK_o^n$ and $u\mapsto\wp_t(\rho_K(u),u)$ is integrable, $\deV(K, \cdot)$ is finite.   Let  $E_i \subset S^{n-1}$, $i\in \N$, be disjoint Borel sets. By \cite[Lemmas 2.3 and 2.4]{LYZActa}, $\pmb{\alpha}^*_K(\cup_iE_i)=\cup_i\pmb{\alpha}^*_K(E_i)$ and the intersection of any two of these sets has $\mathcal{H}^{n-1}$-measure zero.  The dominated convergence theorem then implies that
\begin{eqnarray*}
\deV(K,\cup_{i}E_i) &=& \frac{1}{n}\int_{\cup_{i}\pmb{\alpha}^*_K(E_i)}\frac{\rho_{K}(u)\, \wp_t(\rho_K(u), u)}{\psi(h_{K}(\alpha_K(u)))}\,du\\ &=& \frac{1}{n}\sum_{i=1}^\infty\int_{\pmb{\alpha}^*_K(E_i)}\frac{\rho_{K}(u)\, \wp_t(\rho_K(u), u)}{\psi(h_{K}(\alpha_K(u)))}\,du =\sum_{i=1}^\infty\deV(K,E_i),
\end{eqnarray*}
so $\deV(K,\cdot)$ is countably additive.

Integrals with respect to $\deV(K,\cdot)$ can be calculated as follows. For any bounded Borel function $g: S^{n-1}\rightarrow \R$, we have
\begin{eqnarray}
\int_{\sphere} g(u)\, d\deV(K, u)&=&\frac{1}{n}\int_{\sphere}g(\alpha_K(u))  \frac{\rho_{K}(u)\, \wp_t(\rho_K(u), u) }{ \psi(h_{K}(\alpha_K(u)))}\,du\label{new measue-11-27}\\
&=&\frac{1}{n}\int_{\partial K}g(\nu_{K}(x))\frac{ \langle x, \nu_{K}(x)\rangle}{\psi (\langle x, \nu_{K}(x)\rangle)}\,
|x|^{1-n}\, \wp_t(|x|, \bar{x})\,\,dx,\label{form-11-28}
\end{eqnarray}
where $\bar{x}=x/|x|$.  (Recall our convention that integration on $\partial K$ is denoted by $dx=d\mathcal{H}^{n-1}(x)$.) Relation (\ref{new measue-11-27}) follows immediately from \eqref{relation-11-27}, and (\ref{form-11-28}) follows from the fact that the bi-Lipschitz radial map $r:\partial K\to \sphere$, given by $r(x)=x/|x|$, has Jacobian $Jr(x)=\langle x,\nu_K(x)\rangle |x|^{-n}$ for all regular boundary points, and hence for $\mathcal{H}^{n-1}$-almost all $x\in\partial K$.

If $K$ is strictly convex, then the gradient $\nabla h_K(u)$ of $h_K$ at $u\in\sphere$ equals the unique $x_K(u)\in \partial K$ with outer unit normal vector $u$, and $\nabla h_K(\nu_K(x))=x$ for ${\mathcal{H}}^{n-1}$-almost all $x\in \partial K$. Using this and \cite[Lemma 2.10]{LYZ-Lp}, \eqref{form-11-28} yields
\begin{eqnarray}\label{mar1}
\lefteqn{
\int_{\sphere} g(u)\, d\deV(K, u)}\nonumber\\
&=&
\frac{1}{n}\int_{\sphere}g(u)\,\frac{h_K(u)}
{\psi (h_K(u))}\,|\nabla h_K(u)|^{1-n}\,
\wp_t\left(|\nabla h_K(u)|, \frac{\nabla h_K(u)}{|\nabla h_K(u)|}\right)\,dS(K,u).
\end{eqnarray}

The following result could be proved in the same way as \cite[Lemma~5.5]{LYZ-Lp}, using Weil's Approximation Lemma. Here we provide an argument which avoids the use of this lemma.

\begin{theorem}\label{valprop}
Let $K\in \cK_o^n$, and let $G,\psi$ be as in Definition~\ref{general dual Orlicz curvature measure-11-27}. Then the measure-valued map $K\mapsto \deV(K,\cdot)$ is a valuation on $\cK_o^n$.
\end{theorem}

\begin{proof}
Let $K,L\in \cK_o^n$ be such that $K\cup L\in \cK_o^n$.  It suffices to show that for any bounded Borel function $g: S^{n-1}\rightarrow \R$, we have
\begin{equation}\label{vprop}
I(K\cap L)+I(K\cup L)=I(K)+I(L),
\end{equation}
where $I(M)=\int_{\sphere} g(u)\, d\deV(M, u)$ for $M\in \cK_o^n$.
The sets $K\cap L$, $K\cup L$, $K$, and $L$ can each be partitioned into three disjoint sets, as follows:
\begin{align}
\partial(K\cap L)&=(\partial K\cap \inte L)\cup(\partial L\cap \inte K)\cup (\partial K\cap\partial L),\label{eqa1}\\
\partial(K\cup L)&=(\partial K\setminus L )\cup (\partial L\setminus K )\cup(\partial K\cap\partial L),\label{eqa2}\\
\partial K&=(\partial K\cap \inte L)\cup(\partial K\setminus L)\cup(\partial K\cap\partial L),\label{eqb1}\\
\partial L&=(\partial L\cap \inte K)\cup (\partial L\setminus K )\cup(\partial K\cap\partial L).\label{eqb2}
\end{align}
Let $\bar{x}=x/|x|$.  For $\mathcal{H}^{n-1}$-almost all $x\in\partial(K\cap L)$, we have
\begin{align}
x\in \partial K\cap \inte L \quad &\Rightarrow \quad \nu_{K\cap L}(x)=\nu_K(x){\text{ and }}\rho_{K\cap L}(\bar x)=\rho_K(\bar x),\label{feb181}\\
x\in \partial L\cap \inte K \quad&\Rightarrow \quad \nu_{K\cap L}(x)=\nu_L(x){\text{ and }}\rho_{K\cap L}(\bar x)=\rho_L(\bar x),\label{feb182}\\
x\in \partial K\cap \partial L\quad&\Rightarrow \quad\nu_{K\cap L}(x)=\nu_K(x)=\nu_L(x){\text{ and }}\rho_{K\cap L}(\bar x)=\rho_K(\bar x)=\rho_L(\bar x),\label{feb183}
\end{align}
where the first set of equations in (\ref{feb183}) hold for $x\in\reg(K\cap L)\cap \reg K\cap\reg L$ since $K\cap L\subset K,L$.  Also,
for $\mathcal{H}^{n-1}$-almost all $x\in\partial(K\cup L)$, we have
\begin{align}
x\in \partial K\setminus L \quad &\Rightarrow \quad\nu_{K\cup L}(x)=\nu_K(x){\text{ and }}\rho_{K\cup L}(\bar x)=\rho_K(\bar x),\label{feb184}\\
x\in \partial L\setminus K\quad &\Rightarrow \quad\nu_{K\cup L}(x)=\nu_L(x){\text{ and }}\rho_{K\cup L}(\bar x)=\rho_L(\bar x),\label{feb185}\\
x\in \partial K\cap \partial L\quad &\Rightarrow \quad\nu_{K\cup L}(x)=\nu_K(x)=\nu_L(x){\text{ and }}\rho_{K\cup L}(\bar x)=\rho_K(\bar x)=\rho_L(\bar x),\label{feb186}
\end{align}
where the first set of equations in (\ref{feb186}) hold for $x\in\reg(K\cup L)\cap \reg K\cap\reg L$ since $K,L\subset  K\cup L$.  Now (\ref{vprop}) follows easily from \eqref{form-11-28}, by first decomposing the integrations over $\partial (K\cap L)$ and $\partial (K\cup L)$ into six contributions via \eqref{eqa1} and \eqref{eqa2}, using (\ref{feb181}--\ref{feb186}), and then recombining these contributions via \eqref{eqb1} and \eqref{eqb2}.
\end{proof}

Some particular cases of (\ref{gencdef}) are worthy of mention.  Firstly, with $\wp=\uf$ and general $\psi$, we prefer to write $\ccV(K,E)$ instead of
 $\widetilde{C}_{\uf,\psi}(K,E)$.  Then we have
\begin{equation}\label{general dual Orlicz curvature measure-11-25}
\ccV(K,E)=\frac{1}{n}\int_{\pmb{\alpha}^*_K(E)}\frac{\phi(\rho_K(u)u)\rho_K(u)^n}
{\psi(h_{K}(\alpha_K(u)))}\,du
\end{equation}
and by specializing \eqref{new measue-11-27} and \eqref{form-11-28} we get
\begin{eqnarray*}
\int_{\sphere}g(u)\, d\ccV(K,u) &=&
\frac{1}{n}\int_{\sphere}g(\nu_K(\rho_K(u)u))\frac{\phi(\rho_K(u)u)\rho_K(u)^n}
{\psi(h_{K}(\alpha_K(u)))}\,du\\
&=&\frac{1}{n}\int_{\partial K}g(\nu_K(x))\frac{\langle x,\nu_K(x)\rangle}{\psi(\langle x,\nu_K(x)\rangle)}\phi(x)\, dx
\end{eqnarray*}
for any bounded Borel function $g:\sphere\to \R$. Here we used
\begin{equation}\label{GG}
\wp_t(\rho_K(u), u)=\phi(\rho_K(u)u)\rho_K(u)^{n-1}.
\end{equation}

If we also choose $\psi=1$ and write  $\cV(K,E)$ instead of $\widetilde{C}_{\uf}
(K,E)$, we obtain
$$
\cV(K,E)=
\frac{1}{n}\int_{\pmb{\alpha}^*_K(E)}\phi(\rho_K(u)u)\rho_K(u)^n\,du,
$$
the general dual Orlicz curvature measure introduced in \cite{XY2017-1}, and in particular we see that
\begin{eqnarray}
\int_{S^{n-1}}g(u)\,d\cV(K,u)&=&
\frac{1}{n}\int_{S^{n-1}}g(\alpha_K(u))
\phi(\rho_K(u)u)\rho_K(u)^n\,du\label{integral-general curvature measure 1}\\
&=&
\frac{1}{n}\int_{\partial  K} g(\nu_K(x))\,\phi(x)\,\langle x,\nu_K(x)\rangle\,dx,\nonumber
\end{eqnarray}
as in \cite[Lemma~3.1]{XY2017-1}.

Note that when $\wp=\of$ is given by (\ref{Phidef}), we have $\dveV(K)=\eV(K)$ as in (\ref{general-quermass-11}), in which case $\wp_t(\rho_K(u), u)=-\phi(\rho_K(u)u)\rho_K(u)^{n-1}$ and hence
$\widetilde{C}_{\overline{\Phi},\psi}(K,E)=-\ccV(K,E)$.
Comparing (\ref{gencdef}) and (\ref{new measue-11-27}), and using (\ref{GG}), we see that
\begin{numcases}
{\int_{S^{n-1}}\frac{g(u)}{\psi(h_K(u))}\,d\cV(K,u)=}
-\int_{\sphere} g(u)\, d\widetilde{C}_{\overline{\Phi},\psi}(K, u)\label{meascon}\\
\int_{\sphere} g(u)\, d\widetilde{C}_{\underline{\Phi},\psi}(K, u).\label{meascon2}
\end{numcases}
Taking
$\phi(x)=|x|^{q-n}\rho_Q(x/|x|)^{n-q}$, for some $Q\in \cS_{c+}^n$ and $q\in \R$, and $\psi(t)=t^p$, $p\in \R$, from (\ref{general dual Orlicz curvature measure-11-25}) we get $\ccV(K,E)=\widetilde{C}_{p,q}(K,Q,E)$, where
$$
\widetilde{C}_{p,q}(K,Q,E)=
\frac{1}{n}\int_{\pmb{\alpha}^*_K(E)}
h_{K}(\alpha_K(u))^{-p}\,\rho_K(u)^q\,\rho_Q(u)^{n-q}\,du
$$
is the {\em{$(p,q)$-dual curvature measure of $K$ relative to $Q$}} introduced in \cite[Definition~4.2]{LYZ-Lp}.  The formula \cite[(5.1), p.~114]{LYZ-Lp} or the preceding discussion show that for any bounded Borel function $g: S^{n-1}\rightarrow \R$, we have
\begin{equation}\label{LYZint}
\int_{\sphere} g(u)\, d\widetilde{C}_{p,q}(K,Q,u)=\frac{1}{n}\int_{\sphere}g(\alpha_K(u))\,
h_{K}(\alpha_K(u))^{-p}\,\rho_{K}(u)^q\,\rho_Q(u)^{n-q}\,du.
\end{equation}

\section{General variational formulas for radial Orlicz linear combinations}\label{section4}		
Our main result in this section is the following variational formula for $\dveV$, where $\wp_t(t,u)=\partial \wp(t,u)/\partial t$.

\bt\label{mixed-theorem-12-09}
Let $\wp$ and $\wp_t$ be continuous on $(0, \infty)\times \sphere$ and let $K, L\in \cS_{c+}^n$.

\noindent
{\rm{(i)}} If $\varphi_1,\varphi_2\in \mathcal{I}$ and $(\varphi_1)'_l(1)>0$, then
\begin{equation}\label{variation--1--1--1209}
\lim_{\ee\rightarrow0^+}\frac{\dveV(K_\ee)-\dveV(K)}
{\ee}=\frac{1}{(\varphi_1)'_l(1)}\int_{\sphere} \varphi_{2}\!\left(\frac{\rho_{L}(u)}{\rho_{K}(u)}\right)\rho_K(u)\,
\wp_t(\rho_K(u),u) \,du,
\end{equation}
where $K_{\ee}=K\widetilde{+}_{\varphi,\ee}L$ has radial function $\rho_{\ee}$ given by \eqref{rlc}.  For $\varphi_1,\varphi_2\in \mathcal{D}$,  \eqref{variation--1--1--1209} holds when $(\varphi_1)'_r(1)<0$, with $(\varphi_1)'_l(1)$ replaced by $(\varphi_1)'_r(1)$.

\smallskip

\noindent
{\rm{(ii)}} Let $a\in \R\cup\{-\infty\}$. If $\varphi\in {\mathcal{J}}_a$ and $\varphi'$ is continuous and nonzero on $(0, \infty)$, then for $g\in C(\sphere)$,
$$\lim_{\ee\rightarrow0}\frac{\dveV(\widehat{K}_{\ee})-\dveV(K)}
{\ee}=\int_{\sphere}\frac{g(u)\,\wp_t(\rho_K(u),u) }{\varphi'\left(\rho_{K}(u)\right)}\,du,
$$
where $\widehat{K}_{\ee}$ has radial function $\widehat{\rho}_{\ee}$ given by \eqref{genplus} with $f_0=\rho_K$.
\et

\begin{proof}
\noindent{\rm{(i)}} By (\ref{def-H-volume-12-07}),
\begin{equation}\label{limi}
\lim_{\ee\rightarrow0^+}\frac{\dveV(K_\ee)-\dveV(K)}{\ee}
=\lim_{\ee\rightarrow0^+}\int_{\sphere}\frac{\wp(\rho_{\ee}(u),u)-
\wp(\rho_K(u),u)}{\ee}\,du.
\end{equation}
Also, by (\ref{newlimits}),
\begin{eqnarray*}
\lim_{\ee\rightarrow0^+}\frac{\wp(\rho_\ee(u),u)-
\wp(\rho_K(u),u)}{\ee}&=&
\wp_t(\rho_K(u),u)\lim_{\ee\rightarrow0^+}\frac{\rho_\ee(u)-
\rho_K(u)}{\ee}\\
&=&\frac{1}{(\varphi_1)'_l(1)} \varphi_{2}\!\left(\frac{\rho_{L}(u)}{\rho_{K}(u)}\right)\rho_K(u)\wp_t(\rho_K(u),u),
\end{eqnarray*}
where the previous limit is uniform on $\sphere$.  Therefore (\ref{variation--1--1--1209}) will follow if we show that the limit and integral in (\ref{limi}) can be interchanged.  To this end, assume that $\varphi_1,\varphi_2\in \mathcal{I}$ and $(\varphi_1)'_l(1)>0$; the proof when $\varphi_1,\varphi_2\in \mathcal{D}$ and $(\varphi_1)'_r(1)<0$ is similar.  If $\rho_1(u)=\rho_{\ee}(u)\big|_{\ee=1}$, it is easy to see from \eqref{rlc} that $\rho_K\le \rho_{\ee}\le \rho_1$ on $S^{n-1}$ when $\varepsilon\in(0,1)$.  Since $\wp_t$ is continuous on $(0, \infty)\times \sphere$,
$$\sup\{|\wp_t(t,u)|: \rho_K(u)\le t\le \rho_1(u),~u\in \sphere\}=m_1<\infty.$$
By the mean value theorem and Lemma~\ref{unif-conv-two-functions-2}(i),
$$\left|\frac{\wp(\rho_{\ee}(u),u)-\wp(\rho_{K}(u),u)}{\ee}\right|\le m_2$$
for $0<\ee<1$.  Thus we may apply the dominated convergence theorem in (\ref{limi}) to complete the proof.

\noindent{\rm{(ii)}}  The argument is very similar to that for (i) above.  Since
\begin{equation}\label{limi2}
\lim_{\ee\rightarrow0}\frac{\dveV(\widehat{K}_\ee)-\dveV(K)}{\ee}
=\lim_{\ee\rightarrow0}\int_{\sphere}\frac{\wp(\widehat{\rho}_{\ee}(u),u)-
\wp(\rho_K(u),u)}{\ee}\,du
\end{equation}
we can use (\ref{conv-hat-sum}) instead of (\ref{newlimits}) and need only justify interchanging the limit and integral in (\ref{limi2}). To see that this is valid, suppose that $\varphi \in {\mathcal{J}}_a$ is strictly increasing; the proof is similar when $\varphi$ is strictly decreasing. Then there exists $\ee_0>0$ such that for $\ee\in (-\ee_0, \ee_0)$ and $u\in \sphere$, we have
$$0<b_1(u)={\varphi}^{-1}\left(\varphi\left(\rho_{K}(u)\right)-\ee_0 m_3\right)\le \widehat{\rho}_{\ee}(u)\le {\varphi}^{-1}\left(\varphi\left(\rho_{K}(u)\right)+\ee_0 m_3\right)=b_2(u)<\infty,$$ where $m_3=\sup_{u\in \sphere} |g(u)|<\infty$ due to $g\in C(\sphere)$. Since $\wp_t$ is continuous on $(0, \infty)\times \sphere$,
$$\sup\{|\wp_t(t,u)|: b_1(u)\le t\le b_2(u),~u\in \sphere\}=m_4<\infty.$$
By the mean value theorem and Lemma~\ref{unif-conv-two-functions-2}(ii),
$$\left|\frac{\wp(\widehat{\rho}_{\ee}(u),u)-\wp(\rho_{K}(u),u)}{\ee}\right|\le m_5$$
for $0<\ee<\ee_0$.  Thus we may apply the dominated convergence theorem in (\ref{limi2}) to complete the proof.
\end{proof}

Recall that $\eV$ and $\teV$ are defined by (\ref{general-quermass-11}) and (\ref{general-quermass-22}), respectively.  Note that when $\wp=\of$ or $\uf$, $\wp_t(t,u)=\pm \phi(tu)t^{n-1}$ is continuous on $(0, \infty)\times \sphere$ because $\phi$ is assumed to be continuous.  The following result is then a direct consequence of the previous theorem.

\bc\label{mixed-theorem-1} Let $\phi:\mathbb{R}^n\setminus\{o\}\rightarrow(0, \infty)$ be a continuous function and let $K, L\in \cS_{c+}^n$.

\noindent
{\rm{(i)}} If $\varphi_1,\varphi_2\in \mathcal{I}$ and $(\varphi_1)'_l(1)>0$, then
\begin{numcases}
{\frac{1}{(\varphi_1)'_l(1)}\int_{\sphere}
\phi(\rho_{K}(u)u)\, \varphi_{2}\!\left(\frac{\rho_{L}(u)}{\rho_{K}(u)}\right) \rho_{K}(u)^n \,du=}
\lim_{\ee\rightarrow0^+}\frac{\eV(K)-\eV(\rho_\ee)}
{\ee}\label{variation--1--1--1}\\
\lim_{\ee\rightarrow0^+}\frac{\teV(\rho_\ee)-\teV(K)}
{\ee},\label{variation--1--1--2}
\end{numcases}
where $\rho_{\ee}$ is given by \eqref{rlc}, provided $\of$ (or $\uf$, respectively) is continuous. For $\varphi_1,\varphi_2\in \mathcal{D}$, \eqref{variation--1--1--1} and \eqref{variation--1--1--2} hold when $(\varphi_1)'_r(1)<0$, with $(\varphi_1)'_l(1)$ replaced by $(\varphi_1)'_r(1)$.

\smallskip

\noindent
{\rm{(ii)}} Let $a\in \R\cup\{-\infty\}$. If $\varphi\in {\mathcal{J}}_a$ and $\varphi'$ is continuous and nonzero on $(0, \infty)$, then for all $g\in C(\sphere)$,
$$
\int_{\sphere}\frac{\phi(\rho_{K}(u)u)\, \rho_{K}(u)^{n-1}}{\varphi'\left(\rho_{K}(u)\right)} g(u) \,du=
\begin{dcases}
\lim_{\ee\rightarrow0}\frac{\eV(K) - \eV(\widehat{\rho}_{\ee})}{\ee}\nonumber\\
\lim_{\ee\rightarrow0}\frac{\teV(\widehat{\rho}_{\ee})-
\teV(K)}{\ee},\nonumber
\end{dcases}
$$
where $\widehat{\rho}_{\ee}$ is given by \eqref{genplus} with $f_0=\rho_K$.
\ec

Formulas (\ref{variation--1--1--1}) and (\ref{variation--1--1--2}) motivate the following definition of the {\em general dual Orlicz mixed volume} $\veV(K, L)$. For $K, L\in \cS_{c+}^n$, continuous $\phi: \Rn\setminus\{o\}\rightarrow (0, \infty)$, and continuous $\varphi: (0, \infty)\rightarrow (0, \infty)$, let
\begin{equation}\label{dua1}
\veV(K, L)=\frac{1}{n}\int_{\sphere}\phi(\rho_{K}(u)u)\,\varphi\! \left(\frac{\rho_{L}(u)}{\rho_{K}(u)}\right) \rho_{K}(u)^n \,du.
\end{equation}
Then (\ref{variation--1--1--1}) and (\ref{variation--1--1--2}) become
$$
\veVt(K, L)=
\begin{dcases}
\frac{(\varphi_1)'_l(1)}{n} \lim_{\ee\rightarrow0^+}\frac{\eV(K)-\eV(\rho_\ee)}{\ee}\nonumber\\  \frac{(\varphi_1)'_l(1)}{n} \lim_{\ee\rightarrow0^+}\frac{\teV(\rho_\ee)-\teV(K)}{\ee}.\nonumber
\end{dcases}
$$
The special case of (\ref{variation--1--1--1}) and (\ref{variation--1--1--2}) when $\phi\equiv 1$ was proved in \cite[Theorem~5.4]{ghwy15} (see also \cite[Theorem~4.1]{Zhub2014}) and the corresponding quantity $\veV(K, L)$ was called the {\em Orlicz dual mixed volume}.

On the other hand, Corollary~\ref{mixed-theorem-1}(ii) suggests an alternative definition of the general dual mixed volume. For all $K\in \cS_{c+}^n$, $g\in C(\sphere)$, continuous $\phi: \Rn\setminus\{o\}\rightarrow (0, \infty)$, and continuous $\varphi: (0, \infty)\rightarrow (0, \infty)$, define
\begin{equation}\label{dua2}
\oveV(K, g)=\frac{1}{n}\int_{\sphere} \phi(\rho_{K}(u)u)\, \varphi\! \left(\rho_{K}(u)\right) g(u)\,du.
\end{equation}
Then the formulas in Corollary~\ref{mixed-theorem-1}(ii) can be rewritten as
\begin{numcases}
{\breve{V}_{\phi,\varphi_0}(K, g)=}
\lim_{\ee\rightarrow0}\frac{\eV(K) - \eV(\widehat{\rho}_{\ee})}{\ee}\nonumber\\
\lim_{\ee\rightarrow0}\frac{ \teV(\widehat{\rho}_{\ee})-\teV(K)}{\ee},\nonumber
\end{numcases}
where $\varphi_0(t)=nt^{n-1}/\varphi'(t)$.  In particular, one can define a dual Orlicz mixed volume of $K$ and $L$ by letting $g=\psi(\rho_L)$, where $\psi: (0,\infty)\rightarrow (0, \infty)$ is continuous and $L\in \cS_{c+}^n$, namely
$$\ovepV(K,L)=\frac{1}{n}\int_{\sphere} \phi(\rho_{K}(u)u)\, \varphi\! \left(\rho_{K}(u)\right)\,\psi(\rho_L(u))\,du.$$

Note that both $\veV(K, L)$ and $\oveV(K,g)$ are special cases of $\dveV(K)$, corresponding to setting
$$\wp(t,u)=\frac{1}{n}\phi(tu)\,\varphi\left(\frac{\rho_L(u)}{t}\right)\,t^n$$
or
$$\wp(t,u)=\frac{1}{n}\phi(tu)\,\varphi(t)\,g(u),$$
respectively.

\section{General variational formulas for Orlicz linear combinations}\label{section5}
We shall assume throughout the section that $\Omega\subset  S^{n-1}$ is a closed set not contained in any closed hemisphere of $S^{n-1}$.

Let $h_0,\rho_0\in C^+(\Omega)$ and let $h_{\ee}$ and $\rho_{\ee}$ be defined by (\ref{genplus}) with $f_0=h_0$ and $f_0=\rho_0$, respectively.  In Lemma~\ref{unif-conv-two-functions-2}(ii), we may replace $\rho_K$ by $h_0$ or $\rho_0$ to conclude that $h_{\ee}\rightarrow h_0$ and $\rho_{\ee}\rightarrow \rho_0$ uniformly on $\Omega$.  (In Section~\ref{prel}, $h_{\ee}$ and $\rho_{\ee}$ were denoted by $\widehat{h}_{\ee}$ and $\widehat{\rho}_{\ee}$, but hereafter we omit the hats for ease of notation.) Hence $[h_{\ee}]\rightarrow [h_0]$ and $\langle \rho_{\ee}\rangle \rightarrow \langle\rho_0\rangle$ as $\ee\rightarrow 0$. However, in order to get a variational formula for the general dual Orlicz volume, we shall need the following lemma.  It was proved for $\varphi(t)=\log t$ in \cite[Lemmas~4.1 and~4.2]{LYZActa} and was noted for $t^p$, $p\neq 0$, in the proof of \cite[Theorem~6.5]{LYZ-Lp}.  Recall from Section~\ref{Maps} that $\sphere\setminus\eta_{\langle \rho_0\rangle}$ is the set of regular normal vectors of $\langle \rho_0\rangle\in \cK_o^n$.

\bl \label{variation--11-21}
Let $g\in C(\Omega)$, let $\rho_0\in C^+(\Omega)$, and let $a\in \R\cup\{-\infty\}$. Suppose that $\varphi\in {\mathcal{J}}_a$ is continuously differentiable and such that $\varphi'$ is nonzero on $(0,\infty)$.  For $v\in S^{n-1}\setminus \eta_{\langle \rho_0\rangle}$,
\begin{equation}\label{414}
\lim_{\ee\rightarrow 0}\frac{\log h_{\langle \rho_{\ee}\rangle}(v)-\log h_{\langle \rho_0\rangle}(v)}{\ee}=\frac{g(\alpha_{\langle \rho_0\rangle^*}(v))}{\rho_{0}(\alpha_{\langle \rho_0\rangle^*}(v))\, \varphi'(\rho_{0}(\alpha_{\langle \rho_0\rangle^*}(v)))},
\end{equation}
where $\rho_{\ee}$ is defined by (\ref{genplus}) with $f_0=\rho_0$. Moreover, there exist $\delta, m_0>0$ such that
\begin{equation}\label{unif-bounded-11-21}
|\log h_{\langle \rho_{\ee}\rangle}(v)-\log h_{\langle \rho_0\rangle}(v)|\le m_0 |\ee|
\end{equation}
for $\ee\in (-\delta,\delta)$ and $v\in S^{n-1}$.
\el	

\begin{proof}
We shall assume that $\varphi\in {\mathcal{J}}_a$ is strictly increasing, since the case when it is strictly decreasing is similar. Since $g\in C(\Omega)$, we have $m_1=\sup_{u\in \Omega} |g(u)|<\infty$.
Then there exists $\delta_0>0$ such that for $\ee\in [-\delta_0, \delta_0]$ and $u\in \Omega$,
$$0<{\varphi}^{-1}\left(\varphi\left(\rho_0(u)\right)-\delta_0 m_1\right)\leq \rho_{\ee}(u) \leq {\varphi}^{-1}\left(\varphi\left(\rho_0(u)\right)+\delta_0 m_1\right)$$
and $\inf_{u\in \Omega}|\varphi'(\rho_{\ee}(u))|>0$. For $u\in \Omega$ and $\ee\in (-\delta_0, \delta_0)$, let
$$
H_u(\ee)=\log\rho_{\ee}(u)=\log\left(\varphi^{-1}(\varphi(\rho_0(u))+\ee\, g(u))\right),
$$
from which we obtain
$$
H_u'(\ee)=\frac{g(u)}{\rho_{\ee}(u)\varphi'(\rho_{\ee}(u))}.
$$
By the mean value theorem, for all $u\in\Omega$ and $\ee\in (-\delta_0,\delta_0)$, we get
$$
H_u(\ee)-H_u(0)=\ee\, H_u'(\theta\, \ee),
$$
where $\theta=\theta(u,\ee)\in (0,1)$. In other words,
\begin{equation}\label{hope1}
\log\rho_{\ee}(u)-\log\rho_{0}(u)=\ee\,\frac{g(u)}
{\rho_{\theta(u,\ee)\ee}(u)\,\varphi'(\rho_{\theta(u,\ee)\ee}(u))}
\end{equation}
for $u\in\Omega$ and $\ee\in(-\delta_0,\delta_0)$.

Let $v\in\sphere\setminus \eta_{\langle\rho_{0}\rangle}$.  If $\ee\in (-\delta_0, \delta_0)$, there is a $u_{\ee}\in\Omega$ such that for $u\in\Omega,$
\begin{eqnarray}
h_{\langle\rho_{\ee}\rangle}(v)&=& \langle u_\ee, v\rangle \rho_{\ee}(u_\ee), \ \ \ \  h_{\langle\rho_{\ee}\rangle}(v)\geq \langle u, v\rangle \rho_{\ee}(u) \label{vector-ee-1}, \\
h_{\langle\rho_{0}\rangle}(v)&\geq& \langle u_\ee, v\rangle \rho_{\langle\rho_{0}\rangle}(u_\ee), \ \ \mbox{and}\ \ \rho_{\langle\rho_{0}\rangle}(u_\ee)\geq \rho_{0}(u_\ee). \nonumber \end{eqnarray}
Moreover, $\langle u_\ee, v\rangle>0$ for $\ee\in (-\delta_0, \delta_0)$. Hence, using the equation in (\ref{vector-ee-1}), the inequality in (\ref{vector-ee-1}) with $u=u_{\ee}$, and \eqref{hope1} for $u=u_{\ee}$, we get
\begin{eqnarray}\label{comparison-11-21}
\log h_{\langle\rho_{\ee}\rangle}(v)-\log h_{\langle\rho_{0}\rangle}(v)
&\leq&\log \rho_{\ee}(u_{\ee})-\log \rho_{0}(u_{\ee}) \nonumber \\
&=&\ee\frac{g(u_\ee)}{\rho_{\theta(u_{\ee},\ee)\ee}( u_\ee)\, \varphi'(\rho_{\theta(u_{\ee},\ee)\ee}(u_\ee))}.
\end{eqnarray}	
From the equation in (\ref{vector-ee-1}) with $\ee=0$, the inequality in (\ref{vector-ee-1}) with $u=u_0$, and from \eqref{hope1} with $u=u_0$, we obtain
\begin{eqnarray}
\log h_{\langle\rho_{\ee}\rangle}(v)-\log h_{\langle\rho_{0}\rangle}(v)&=&\log h_{\langle\rho_{\ee}\rangle}(v)-\log \rho_{0}(u_0)-\log\langle u_0, v\rangle \nonumber\\
&\geq&\log \rho_{\ee}(u_{0})-\log \rho_{0}(u_{0}) \nonumber \\
&=&\ee\frac{g(u_0)}{\rho_{\theta(u_{0},\ee)\ee}( u_0)\, \varphi'(\rho_{\theta(u_{0},\ee)\ee}(u_0))}. \label{comparison-11-21-1}
\end{eqnarray}
Exactly as in \cite[(4.7), (4.8)]{LYZActa}, we have $u_0=\alpha^*_{\langle \rho_0\rangle}(v)=\alpha_{\langle \rho_0\rangle^*}(v)$ and  $\lim_{\ee\rightarrow0} u_{\ee}=u_{0}$. Since $g$ is continuous and $u_\ee\to u_0$, we get $g(u_\ee)\to g(u_0)$ as $\ee\to 0$. From $\theta(\cdot)\in(0,1)$ it follows that $\theta(\cdot)\ee\to 0$  as $\ee\to 0$.  Moreover, $\rho_{\theta(\cdot)\ee}(u_\ee)=\varphi^{-1}(\varphi(\rho_0(u_\ee))+\theta(\cdot)\ee g(u_\ee))\to \rho_0(u_0)$ and, similarly, $\rho_{\theta(\cdot)\ee}(u_0)\to \rho_0(u_0)$ as $\ee\to 0$.
Thus we  conclude that
\begin{eqnarray*}
\lim_{\ee\rightarrow 0}\frac{\log h_{\langle \rho_{\ee}\rangle}(v)-\log h_{\langle \rho_0\rangle}(v)}{\ee}=\frac{g(u_0)}{\rho_{0}(u_0)\, \varphi'(\rho_{0}(u_0))}.
\end{eqnarray*}
Substituting $u_0=\alpha^*_{\langle \rho_0\rangle}(v)$, we obtain (\ref{414}).

If $\delta_0$ is sufficiently small, then \eqref{comparison-11-21}  and  \eqref{comparison-11-21-1} imply that
if $v\in\sphere\setminus \eta_{\langle\rho_{0}\rangle}$ then
$$
\left|\log h_{\langle\rho_{\ee}\rangle}(v)-\log h_{\langle\rho_{0}\rangle}(v)\right| \le
|\ee|\, \sup_{u\in \Omega,\,\theta\in [0,1]}
\left|\frac{g(u)}{\rho_{\theta\ee}(u)\, \varphi'(\rho_{\theta\ee}(u))}\right|=m_2 |\ee|,
$$
say, for some $m_2<\infty$.  From this, we see that (\ref{unif-bounded-11-21}) holds for $v\in \sphere\setminus\eta_{\langle \rho_0\rangle}$ and hence, by (\ref{zerosets}) and the continuity of support functions, for $v\in \sphere$.
\end{proof}

\bl\label{ovev-11-27}
Let $g\in C(\Omega)$, let $h_0\in C^+(\Omega)$, and let $a\in \R\cup\{-\infty\}$. Suppose that $\varphi\in {\mathcal{J}}_a$ is continuously differentiable and such that $\varphi'$ is nonzero on $(0,\infty)$. If $\wp$ and $\wp_t$ are continuous on $(0, \infty)\times \sphere$, then
\begin{equation}\label{444}
\lim_{\ee\rightarrow 0} \frac{\dveV([h_{\ee}])-\dveV([h_{0}])}{\ee}=\int_{\sphere\setminus \eta_{\langle \kappa_0\rangle }} J(0, u)  \frac{\kappa_{0}(\alpha_{\langle \kappa_0\rangle^*}(u)) \, g(\alpha_{\langle \kappa_0\rangle^*}(u))}{\varphi'(\kappa_{0}(\alpha_{\langle \kappa_0\rangle^*}(u))^{-1})}\,du,
\end{equation}
where $h_{\ee}$ is given by \eqref{genplus} with $f_0=h_0$, and for $\ee$ sufficiently close to $0$, $\kappa_{\ee}=1/h_{\ee}$ and
\begin{equation}\label{Jdef}
J(\ee,u)=\rho_{\langle\kappa_{\ee}\rangle^*}(u)\, \wp_{t}(\rho_{\langle\kappa_{\ee}\rangle^*}(u),u).
\end{equation}
\el

\begin{proof}
Let $\overline{\varphi}(t)=\varphi(1/t)$ for all $t\in (0,\infty)$. Clearly $\overline{\varphi}\in {\mathcal{J}}_a$. Also, for $t\in (0, \infty)$, we have $\overline{\varphi}'(t)=-t^{-2}\varphi'(1/t)$. Hence $\overline{\varphi}$ satisfies the conditions for $\varphi$ in Lemma~\ref{variation--11-21}.  It is easy to check that
$$\kappa_{\ee}(u)= {\overline{\varphi}}^{-1}\left(\overline{\varphi}\left
(\kappa_0(u)\right)+\ee g(u) \right),
$$
that is, $\kappa_{\ee}$ is given by (\ref{genplus}) when $\varphi$ and $f_0$ are replaced by $\overline{\varphi}$ and $\kappa_0$.  By (\ref{414}), with $\rho_{\ee}$ and $\varphi$ replaced by $\kappa_{\ee}$ and ${\overline{\varphi}}$, respectively, for sufficiently small $|\ee|$, we obtain, for $u\in S^{n-1}\setminus \eta_{\langle \kappa_0\rangle}$,
\begin{eqnarray}
\lim_{\ee\rightarrow 0}\frac{\log \rho_{\langle \kappa_{ \ee}\rangle^*}(u)-\log \rho_{\langle \kappa_0\rangle^*}(u)}{\ee}& =&
-\lim_{\ee\rightarrow 0}\frac{\log h_{\langle \kappa_{ \ee}\rangle}(u)-\log h_{\langle \kappa_0\rangle}(u)}{\ee}\nonumber\\
&=&-\frac{g(\alpha_{\langle \kappa_0\rangle^*}(u))}{\kappa_{0}(\alpha_{\langle \kappa_0\rangle^*}(u))\,\overline{\varphi}'(\kappa_{0}(\alpha_{\langle \kappa_0\rangle^*}(u)))}\nonumber\\
&=&
\frac{\kappa_{0}(\alpha_{\langle \kappa_0\rangle^*}(u)) \, g(\alpha_{\langle \kappa_0\rangle^*}(u))}{\varphi'(\kappa_{0}(\alpha_{\langle \kappa_0\rangle^*}(u))^{-1})}.\label{unif-limit-11-27}
\end{eqnarray}
Moreover, comparing (\ref{unif-bounded-11-21}), there exist $\delta, m_0>0$ such that \begin{equation}\label{unif-bounded-11-27}
|\log h_{\langle \kappa_{ \ee}\rangle}(u)-\log h_{\langle \kappa_0\rangle}(u)|\le m_0 |\ee|
\end{equation}
for $\ee\in (-\delta,\delta)$ and $u\in S^{n-1}$.	

Note that
\begin{eqnarray} \label{inequality bound -11-27}
\frac{\,d \wp(\rho_{\langle \kappa_{\ee}\rangle^*}(u),u)}{\,d\ee}=
\wp_t(\rho_{\langle\kappa_{ \ee}\rangle^*}(u),u)\,\frac{d}{d\ee} \rho_{\langle\kappa_{ \ee}\rangle^*}(u)=
J(\ee, u)\,\frac{d}{d\ee}\log \rho_{\langle\kappa_{ \ee}\rangle^*}(u).
\end{eqnarray}
By our assumptions, there exists $0<\delta_1\le \delta$ and $m_1>0$ such that $|J(\ee, u)|<m_1$ for $\ee\in (-\delta_1, \delta_1)$ and $u\in \sphere$. It follows from (\ref{unif-bounded-11-27}), (\ref{inequality bound -11-27}), and the mean value theorem  that, for $\ee\in (-\delta_1, \delta_1)$ and $u\in \sphere$,
$$\left|\frac{\wp(\rho_{\langle \kappa_{\ee}\rangle^*}(u), u)-\wp(\rho_{\langle \kappa_{0}\rangle^*}(u), u)}{\ee}\right|<m_0m_1.
$$
From (\ref{relation}), we know that $[h_{\ee}]=\langle \kappa_{\ee} \rangle^*$, so $\langle\kappa_{\ee}\rangle^*\to \langle\kappa_{0}\rangle^*$ as $\ee\to 0$. By the dominated convergence theorem, (\ref{unif-limit-11-27}), and  (\ref{inequality bound -11-27}), we obtain
\begin{eqnarray*}
\lim_{\ee\rightarrow 0} \frac{\dveV([h_{\ee}])-\dveV([h_{0}])}{\ee}
&=&  \lim_{\ee\rightarrow 0}  \int_{\sphere}\frac{\wp(\rho_{\langle \kappa_{\ee}\rangle^*}(u), u)-\wp(\rho_{\langle \kappa_{0}\rangle^*}(u), u)}{\ee}  \,du\\&=&  \int_{\sphere} \lim_{\ee\rightarrow 0}   \frac{\wp(\rho_{\langle \kappa_{\ee}\rangle^*}(u), u)-\wp(\rho_{\langle \kappa_{0}\rangle^*}(u), u)}{\ee}  \,du\\
&=&\int_{\sphere\setminus \eta_{\langle \kappa_0\rangle }} J(0, u)  \frac{\kappa_{0}(\alpha_{\langle \kappa_0\rangle^*}(u))\, g(\alpha_{\langle \kappa_0\rangle^*}(u))}{\varphi'(\kappa_{0}(\alpha_{\langle \kappa_0\rangle^*}(u))^{-1})}\,du,
\end{eqnarray*}
where we have used the fact that ${\mathcal{H}}^{n-1}(\eta_{\langle\kappa_0\rangle})=0$ by (\ref{zerosets}).
\end{proof}

The next theorem will be used in the proof of Theorem~\ref{solution-general-dual-Orlicz-main theorem-11-27}.  It generalizes previous results of this type, which originated with \cite[Theorem~4.5]{LYZActa}; see the discussion after Corollary~\ref{variational-for-decreaing-1}.

\bt\label{ovev-cor}
Let $g\in C(\Omega)$, let $h_0\in C^+(\Omega)$, and let $a\in \R\cup\{-\infty\}$. Suppose that $\varphi\in {\mathcal{J}}_a$ is continuously differentiable and such that $\varphi'$ is nonzero on $(0,\infty)$. If $\wp$ and $\wp_t$ are continuous on $(0, \infty)\times \sphere$, then
\begin{equation}\label{variation-11-27-1}
\lim_{\ee\rightarrow 0} \frac{\dveV([h_{\ee}])-\dveV([h_{0}])}{\ee}=
n\int_{\Omega} g(u)\, d\deV([h_{0}], u),
\end{equation}
where $h_{\ee}$ is given by \eqref{genplus} with $f_0=h_0$, and $\psi(t)=t\varphi'(t)$.
\et

\begin{proof}
It follows from \cite[p.~364]{LYZActa} that there exists a continuous function $\overline{g}: S^{n-1}\rightarrow \R$, such that, for $u\in S^{n-1}\setminus \eta_{\langle\kappa_0\rangle}$, $$
g(\alpha_{\langle\kappa_0\rangle^*}(u))=
(\overline{g}1_{\Omega})(\alpha_{\langle \kappa_0\rangle^*}(u)).
$$
Using this, $\kappa_0=1/h_0$, the relation $\langle \kappa_{0} \rangle^*=[h_{0}]$ given by (\ref{relation}), (\ref{map-reverse}), (\ref{Jdef}) with $\ee=0$, ${\mathcal{H}}(\eta_{\langle\kappa_0\rangle})=0$ from (\ref{zerosets}), and (\ref{new measue-11-27}), the formula (\ref{444}) becomes
\begin{eqnarray*}
\lim_{\ee\rightarrow 0} \frac{\dveV([h_{\ee}])-\dveV([h_{0}])}{\ee}
&=&
\int_{\sphere\setminus \eta_{\langle \kappa_0\rangle }}
\frac{(\overline{g}1_{\Omega})(\alpha_{[h_0]}(u))\,\rho_{[h_0]}(u)\, \wp_{t}(\rho_{[h_0]}(u),u)}
{h_{0}(\alpha_{[h_0]}(u))\, \varphi'(h_{0}(\alpha_{[h_0]}(u)))}\,du\\
&=&
\int_{\sphere}
\frac{(\overline{g}1_{\Omega})(\alpha_{[h_0]}(u))\,\rho_{[h_0]}(u)\, \wp_{t}(\rho_{[h_0]}(u),u)}
{\psi(h_{0}(\alpha_{[h_0]}(u)))}\,du\\
&=&
n\int_{\Omega} g(u)\, d\deV([h_{0}], u),
\end{eqnarray*}
where we also used the fact that
$$
h_{[h_0]}(\alpha_{[h_0]}(u))=h_0(\alpha_{[h_0]}(u))\ \ \ \mathrm{for}~ \mathcal{H}^{n-1}\textrm{-almost~all}~u\in \sphere.
$$
To see this, note that for $\mathcal{H}^{n-1}$-almost all $u\in\sphere$, we have $\alpha_{[h_0]}(u)=\nu_{[h_0]}(\rho_{[h_0]}(u)u)$ and $\rho_{[h_0]}(u)u$ is a regular boundary point of $[h_0]$.  The rest is done by the proof of Lemma~7.5.1 in \cite[p.~411]{Sch}, which shows that if $x\in \partial [h_0]$ is a regular boundary point, then $h_{[h_0]}(\nu_{[h_0]}(x))=h_0(\nu_{[h_0]}(x))$.
\end{proof}

\br\label{remfeb25}
{\em It is possible to extend the definition (\ref{def-H-volume-12-07}) of the general dual volume $\dveV(K)$ by allowing continuous functions $G:(0,\infty)\times\sphere\to\R$.  In this case, of course, $\dveV(K)$ may be negative, but the extended definition has the advantage of including fundamental concepts such as the dual entropy $\widetilde{E}(K)$ of $K$.  This is defined by
$$\widetilde{E}(K)=\frac{1}{n}\int_{\sphere}\log\rho_K(u)\,du,$$
corresponding to taking $G(t,u)=(1/n)\log t$ in (\ref{def-H-volume-12-07}).  Definition~\ref{general dual Orlicz curvature measure-11-27} of the measure $\widetilde{C}_{G,\psi}$ and the integral formulas (\ref{new measue-11-27}) and (\ref{form-11-28}) remain valid for continuous functions $G:(0,\infty)\times\sphere\to\R$, as do Theorems~\ref{valprop}, \ref{mixed-theorem-12-09}, and \ref{ovev-cor}, as well as Theorem~\ref{variational-for-decreaing-12-10} below.}
\er

Theorem~\ref{ovev-cor} and its extended form indicated in Remark~\ref{remfeb25} may be used to retrieve the formulas in \cite[Theorem~6.5]{LYZ-Lp}, which in turn generalize those in \cite[Corollary~4.8]{LYZActa}.  To see this, let $K, L\in \cK^n_o$ and let $\varphi(t)=t^p$, $p\neq 0$.  Setting $h_0=h_K$ and $g=h_L^p$, we see from \eqref{genplus} with $f_0=h_0$ that $[h_{\ee}]=K\widehat{+}_p\,\ee\cdot L$, the $L_p$ linear combination of $K$ and $L$.  Taking $G(t,u)=(1/n)t^{q}\,\rho_Q(u)^{n-q}$, for some $Q\in \cS_{c+}^n$ and $q\neq 0$, where $t>0$ and $u\in \sphere$, we have $\dveV(K)=\widetilde{V}_q(K,Q)$ as in (\ref{qmixedv}).  With $\Omega=S^{n-1}$ and $\psi(t)=t\varphi'(t)=pt^p$, and using (\ref{new measue-11-27}) and (\ref{LYZint}), we obtain
\begin{eqnarray*}
n\int_{\Omega} g(u)\, d\deV([h_{0}], u)&=& n\int_{\sphere}h_L(u)^p
\,d\deV(K,u)\\
&=& \frac{q}{np}\int_{\sphere}\left(\frac{h_L(\alpha_K(u))}
{h_K(\alpha_K(u))}\right)^p\rho_K(u)^q\rho_Q(u)^{n-q}\,du\\
&=&\frac{q}{p}\int_{\sphere}h_L(u)^p\,d\widetilde{C}_{p,q}(K,Q,u).
\end{eqnarray*}
Thus (\ref{variation-11-27-1}) becomes
$$
\lim_{\ee\rightarrow 0} \frac{\widetilde{V}_q(K\widehat{+}_p\,\ee\cdot L, Q)-\widetilde{V}_q(K, Q)}{\ee}=
\frac{q}{p}\int_{\sphere}h_L(u)^p\,d\widetilde{C}_{p,q}(K,Q,u),
$$
the formula in \cite[(6.3), Theorem~6.5]{LYZ-Lp} (where $\widehat{+}_p$ is denoted by $+_p$; in our usage, the two are equivalent for $p\ge 1$, when $h_{\ee}$ above is a support function).  Next, we take instead $\varphi(t)=\log t$ and $g=\log h_L$, noting from \eqref{genplus} with $f_0=h_0$ that $[h_{\ee}]=K\widehat{+}_0\,\ee\cdot L$, the logarithmic linear combination of $K$ and $L$.  Then, again with $\Omega=S^{n-1}$ and $\psi(t)=t\varphi'(t)=1$, an argument similar to that above shows that (\ref{variation-11-27-1}) becomes
$$
\lim_{\ee\rightarrow 0} \frac{\widetilde{V}_q(K\widehat{+}_0\,\ee\cdot L, Q)-\widetilde{V}_q(K, Q)}{\ee}=
q\int_{\sphere}\log h_L(u)\,d\widetilde{C}_{0,q}(K,Q,u),
$$
the formula in \cite[(6.4), Theorem~6.5]{LYZ-Lp} (where $\widehat{+}_0$ is denoted by $+_0$).

If instead we take $G(t,u)=(1/n)\log(t/\rho_Q(u))\,\rho_Q(u)^{n}$, for some $Q\in \cS_{c+}^n$, where $t>0$ and $u\in \sphere$, we have
$$\dveV(K)=\frac{1}{n}\int_{\sphere}\log\left(\frac
{\rho_K(u)}{\rho_Q(u)}\right)\rho_Q(u)^n\,du=\widetilde{E}(K,Q),$$
the dual mixed entropy of $K$ and $Q$.  Then similar computations to those above show that (\ref{variation-11-27-1}) (now justified via Remark~\ref{remfeb25}) yield the variational formulas
\cite[(6.5) and (6.6), Theorem~6.5]{LYZ-Lp} for $\widetilde{E}(K,Q)$.

The following corollary is a direct consequence of the previous theorem with $G=\overline{\Phi}$ or $\underline{\Phi}$, and (\ref{meascon}) and (\ref{meascon2}) with $\psi(t)=t\varphi'(t)$.  When $\varphi(t)=\log t$, it was proved in \cite[Theorem~4.1]{XY2017-1}.

\bc\label{variational-for-decreaing-1}
Let $g\in C(\Omega)$, let $h_0\in C^+(\Omega)$, and let $a\in \R\cup\{-\infty\}$. Suppose that $\varphi\in {\mathcal{J}}_a$ is continuously differentiable and such that $\varphi'$ is nonzero on $(0,\infty)$. If $\phi:\R^n\setminus\{o\}\to (0,\infty)$ and $\overline{\Phi}$ (or $\underline{\Phi}$, as appropriate) are continuous, then
\begin{eqnarray}\label{nnn}
n\int_\Omega \frac{g(u)}{h_0(u)\,\varphi'(h_{0}(u))}\,d\cV([h_0],u)=
\begin{dcases}
\lim_{\ee\rightarrow 0} \frac{\eV([h_0])-\eV([h_{\ee}])}{\ee}\\
\lim_{\ee\rightarrow 0} \frac{\teV([h_{\ee}])-\teV([h_0])}{\ee},
\end{dcases}
\end{eqnarray}
where $h_{\ee}$ is given by \eqref{genplus} with $f_0=h_0$.
\ec

The following version of Theorem~\ref{ovev-cor} for Orlicz linear combination of the form (\ref{Orlicz-addition--1-1}) can be proved in a similar fashion.  We omit the proof. Recall that $\leV([h_1], \cdot)=\deV([h_1], \cdot)$ when $\psi\equiv 1$, as in Definition~\ref{general dual Orlicz curvature measure-11-27}.

\bt\label{variational-for-decreaing-12-10}
Let $h_1, h_2\in C^+(\Omega)$ and let $\varphi_1, \varphi_2 \in \mathcal{I}$ or $\varphi_1, \varphi_2\in \mathcal{D}$.  Suppose that for $i=1,2$, $\varphi_i$ is continuously differentiable and such that $\varphi_i'$ is nonzero on $(0,\infty)$.  If $\wp$ and $\wp_t$ are continuous on $(0, \infty)\times \sphere$, then
\begin{eqnarray*}
\lim_{\ee\rightarrow 0^+} \frac{\dveV([h_{\ee}])-\dveV([h_1])}{\ee} =  \frac{n}{\varphi_{1}'(1)}\int_\Omega \varphi_{2}\!\left(\frac{h_2(u)}{h_1(u)}\right)\,d\leV([h_1],u),
\end{eqnarray*}
where $h_{\ee}$ is given by \eqref{Orlicz-addition--1-1} with $h_K$ and $h_L$ replaced by $h_1$ and $h_2$, respectively.
\et

Again, the following corollary is a direct consequence of the previous theorem with $G=\overline{\Phi}$ or $\underline{\Phi}$.

\bc\label{variational-for-decreaing-12-10-cor}
Let $h_1, h_2\in C^+(\Omega)$ and let $\varphi_1, \varphi_2 \in \mathcal{I}$ or $\varphi_1, \varphi_2\in \mathcal{D}$.  Suppose that for $i=1,2$, $\varphi_i$ is continuously differentiable and such that $\varphi_i'$ is nonzero on $(0,\infty)$.  If $\phi:\R^n\setminus\{o\}\to (0,\infty)$ and $\overline{\Phi}$ (or $\underline{\Phi}$, as appropriate) are continuous, then
\begin{eqnarray*}
\frac{n}{\varphi_{1}'(1)}\int_\Omega \varphi_{2}\!\left(\frac{h_2(u)}{h_1(u)}\right)\,d\cV([h_1],u)=
\begin{dcases}
\lim_{\ee\rightarrow 0^+} \frac{\eV([h_1])-\eV([h_{\ee}])}{\ee}\\
\lim_{\ee\rightarrow 0^+} \frac{\teV([h_{\ee}])-\teV([h_1])}{\ee},
\end{dcases}
\end{eqnarray*}
where $h_{\ee}$ is given by \eqref{Orlicz-addition--1-1} with $h_K$ and $h_L$ replaced by $h_1$ and $h_2$, respectively.
\ec
	
\section{Minkowski-type problems} \label{section-6}
This section is dedicated to providing a partial solution to the Orlicz-Minkowski problem for the measure $\deV(K, \cdot)$.

\bl\label{continuity-general-dual-qu-11-27} Let $\wp: (0, \infty) \times \sphere\rightarrow (0, \infty)$ be continuous.  If $K_i\in \cK_o^n$, $i\in \N$, and $K_i\to K\in\cK_o^n$ as $i\to\infty$, then
$\lim_{i\rightarrow \infty}\dveV(K_{i})=\dveV(K)$.
\el

\begin{proof}
Since $K_i\to K\in\cK_o^n$, $\rho_{K_i}\rightarrow \rho_K$ uniformly on $\sphere$. By the continuity of $\wp$, we have $\lim_{i\rightarrow \infty} \wp(\rho_{K_i}(u),u)=\wp(\rho_{K}(u),u)$ and
$\sup\{\wp(\rho_{K_i}(u),u): i\in \N, u\in \sphere\}<\infty$.  It follows from the dominated convergence theorem that
\begin{eqnarray*}
\lim_{i\rightarrow \infty}\dveV(K_{i})=\lim_{i\rightarrow\infty} \int_{\sphere}\wp(\rho_{K_i}(u), u)\,du= \int_{\sphere}\lim_{i\rightarrow\infty}\wp(\rho_{K_i}(u), u)\,du=\dveV(K). \end{eqnarray*}
\end{proof}

\bp\label{prop 11-28--1}
Let $\wp$ and $\wp_t$ be continuous on $(0, \infty)\times \sphere$, let $\psi: (0,\infty)\rightarrow (0,\infty)$ be continuous, and let $K\in \cK_o^n$. The following statements hold.

\noindent{\rm{(i)}} The signed measure $\deV(K, \cdot)$ is absolutely continuous with respect to $S(K,\cdot)$.

\noindent {\rm{(ii)}} If $K_i\in \cK_o^n$, $i\in \N$, and $K_i\to K\in\cK_o^n$ as $i\to\infty$, then $\deV(K_i,\cdot)\rightarrow \deV(K,\cdot)$ weakly.

\noindent {\rm{(iii)}} If $\wp_t>0$ on $(0, \infty)\times\sphere$ (or $\wp_t<0$ on $(0, \infty)\times\sphere$), then $\deV(K,\cdot)$ (or $-\deV(K,\cdot)$, respectively) is a nonzero finite Borel measure not concentrated on any closed hemisphere.
\ep

\begin{proof}
(i) Let $E\subset S^{n-1}$ be a Borel set such that $S(K,E)=0$.  If $g=1_{E}$, the left-hand side of (\ref{new measue-11-27}) is $\deV(K,E)$.  This equals the expression in (\ref{form-11-28}), in which we observe that since $K\in \cK_o^n$, for $x\in\partial K$ both $|x|$ and $\langle x,\nu_{K}(x)\rangle=h_K(\nu_{K}(x))$ are bounded away from zero and bounded above, and hence our assumptions imply that
$$\sup_{x\in \partial K}\left|\frac{\rho_{K}(\bar{x})\, \wp_t(\rho_K(\bar{x}), \bar{x})\,\langle x,\nu_{K}(x)\rangle}{\psi (\langle x, \nu_{K}(x)\rangle)\, |x|^{n}}\right|=c<\infty,$$
where $\bar{x}=x/|x|$.  Then from (\ref{new measue-11-27}) and (\ref{form-11-28}) we conclude, using (\ref{surface:area:1}), that
$$\left|\deV(K,E)\right|\le c\int_{\partial K}1_{E}(\nu_K(x))\,dx=c\,{\mathcal{H}}^{n-1}(\nu_K^{-1}(E))=c\,S(K,E)=0.$$

(ii) Let $g:\sphere\to \R$ be continuous and let
$$I_K(u)=g(\alpha_K(u))\frac{\rho_{K}(u)\, \wp_t(\rho_K(u),u)}{\psi(h_{K}(\alpha_K(u)))}$$
be the integrand of the right-hand side of (\ref{new measue-11-27}). Suppose that $K_i\in \cK_o^n$, $i\in \N$, and $K_i\to K\in\cK_o^n$.  By \cite[Lemma~2.2]{LYZActa}, $\alpha_{K_i}\rightarrow \alpha_K$ and hence, by the continuity of $\wp_t$ and  the continuity of the map $(K,u)\mapsto h_K(u)$ (see \cite[Lemma 1.8.12]{Sch}), $I_{K_i}\to I_K$, ${\mathcal{H}}^{n-1}$-almost everywhere on $\sphere$.  Moreover, our assumptions clearly yield $\sup\{I_{K_i}(u): i\in \N, u\in \sphere\}<\infty$.  It follows from (\ref{new measue-11-27}) and the dominated convergence theorem that
$$\int_{\sphere}g(u)\,d\deV(K_i,u)\to \int_{\sphere}g(u)\,d\deV(K,u)$$
as $i\to\infty$, as required.

(iii)  Suppose that $\wp_t>0$ on $(0, \infty)\times\sphere$; the case when $\wp_t<0$ on $(0, \infty)\times\sphere$ is similar.  Let $m=\min_{x\in\partial K}J_K(x)$, where
$$J_K(x)=\frac{\rho_{K}(\bar{x})\, \wp_t(\rho_K(\bar{x}), \bar{x})\, \langle x, \nu_{K}(x)\rangle}{\psi (\langle x, \nu_{K}(x)\rangle)\, |x|^{n}},\quad x\in\partial K,$$
and $\bar{x}=x/|x|$.  Since $K\in \cK_o^n$, our assumptions imply that $m>0$.  By (\ref{new measue-11-27}) and (\ref{form-11-28}),
\begin{eqnarray*}
\int_{\sphere}\langle u,v\rangle_+\, d\deV(K, v)&=&
\int_{\partial K}
\langle u,\nu_{K}(x)\rangle_+\,J_K(x)\,dx\\
&\ge & m \int_{\partial K}
\langle u,\nu_{K}(x)\rangle_+\,dx=m\int_{\sphere}\langle u,v\rangle_+\,dS(K,v)>0,
\end{eqnarray*}
because $S(K, \cdot)$ satisfies (\ref{condition for Minkowski problem}). This shows that $\deV(K,\cdot)$ also satisfies (\ref{condition for Minkowski problem}).
\end{proof}

In view of Proposition~\ref{prop 11-28--1}(iii), one can ask the following Minkowski-type problem for the signed measure $\deV(\cdot, \cdot)$.

\begin{problem}\label{Minkowski-c-11-28}
For which nonzero finite Borel measures $\mu$ on $\sphere$ and continuous functions $\wp:(0, \infty)\times \sphere\rightarrow (0, \infty)$ and $\psi:(0,\infty)\to(0,\infty)$ do there exist $\tau\in \R$ and  $K\in \cK_{o}^n$ such that $\mu=\tau\,\deV(K,\cdot)$?
\end{problem}

It follows immediately from (\ref{mar1}), on using \cite[(2.2), p.~93 and (3.28), p.~106]{LYZ-Lp}, that solving Problem~\ref{Minkowski-c-11-28} requires finding an
$h:\sphere\to (0,\infty)$ and $\tau\in \R$ that solve (in the weak sense) the Monge-Amp\`{e}re equation
\begin{equation}\label{new2}
\frac{\tau h}{\psi\circ h}\,P(\bar{\nabla}h+h\iota)
\,\det(\bar{\nabla}^2h+hI)=f,
\end{equation}
where $P(x)=|x|^{1-n}G_t(|x|,\bar{x})$ for $x\in \R^n$.  Here $f$ plays the role of the density function of the measure $\mu$ in Problem~\ref{Minkowski-c-11-28} if $\mu$ is absolutely continuous with respect to spherical Lebesgue measure.  Formally, then, Problem~\ref{Minkowski-c-11-28} is more difficult, since it calls for $h$ in (\ref{new2}) to be the support function of a convex body and also a solution for measures that may not have a density function $f$.

To see that (\ref{new2}) is more general than (\ref{LYZPDE}), note firstly that the homogeneity of the left-hand side of (\ref{LYZPDE}) allows us to set $\tau=1$, without loss of generality (if $p\neq q$, which is true in the case $p>0$, $q<0$ of particular interest in the present paper).  Let $p,q\in \R$ and let $Q\in{\mathcal{S}}^n_{c+}$.
For $t>0$ and $u\in\sphere$, we set $\psi(t)=t^p$ and $G(t,u)=(1/q)t^{q}\rho_Q(u)^{n-q}$, if $q\neq 0$, and $G(t,u)=(\log t)\rho_Q(u)^{n}$, otherwise. (When $q\le 0$, we have $G:(0,\infty)\times\sphere\to\R$ and Remark~\ref{remfeb25} applies.)  Then, using the fact that $\rho_Q$ is homogeneous of degree $-1$, we have $P(x)=\rho_Q(x)^{n-q}$, for $q\in \R$ and $x\in \R^n\setminus\{o\}$.  Therefore (\ref{new2}) becomes
$$
h^{1-p}\,\|\bar{\nabla}h+h\iota \|_Q^{q-n}\det(\bar{\nabla}^2h+hI)=f,
$$
where $\|\cdot\|_Q=1/\rho_Q$ is the gauge function of $Q$.  Note that $\|\cdot\|_Q$ is an $n$-dimensional Banach norm if $Q$ is convex and origin symmetric.

Our contribution to Problem~\ref{Minkowski-c-11-28} is as follows.
For the statement and proof of the result, we define
$$
\Sigma_{\varepsilon}(v)=\{u\in S^{n-1}:\langle u, v\rangle\ge \varepsilon\}
$$
for $v\in\sphere$ and $\varepsilon\in (0,1)$.

\bt \label{solution-general-dual-Orlicz-main theorem-11-27}
Let $\mu$ be a nonzero finite Borel measure on $S^{n-1}$ not concentrated on any closed hemisphere.  Let $\wp$ and $\wp_t$ be continuous on $(0, \infty)\times \sphere$ and let $\wp_t<0$ on $(0,\infty)\times\sphere$.  Let $0<\ee_0<1$ and suppose that for $v\in \sphere$ and $0<\varepsilon\le \varepsilon_0$,
\begin{equation}\label{condE2}
\lim_{t\to 0+} \int_{\Sigma_{\varepsilon}(v)} \wp(t,u)\,du=\infty \qquad\text{and}\qquad
\lim_{t\to \infty} \int_{\sphere} \wp(t,u)\,du=0.
\end{equation}
Let $\psi: (0,\infty)\rightarrow (0,\infty)$ be continuous and satisfy
\begin{equation}\label{feb22}
\int_{1}^\infty\frac{\psi(s)}{s}\, ds=\infty.
\end{equation}
Then there exists $K\in \cK_o^n$ such that
\begin{equation}\label{msol}
\frac{\mu}{|\mu|}=\frac{\deV(K, \cdot)}{\deV(K, S^{n-1})}.
\end{equation}
\et

\begin{proof}
Note that the limits in \eqref{condE2} exist, since $t\mapsto G(t,u)$ is decreasing. Define
\begin{equation}\label{eqvarphi}
\varphi(t)=\int_1^t\frac{\psi(s)}{s}\, ds,\quad t>0,
\end{equation}
and
\begin{equation}\label{feb222}
a=-\int_0^1 \frac{\psi(s)}{s}\, ds\in \R\cup\{-\infty\}.
\end{equation}
Then, by (\ref{feb22}), (\ref{eqvarphi}), and (\ref{feb222}), $\varphi\in \mathcal{J}_a$ is strictly increasing and continuously differentiable with $t\varphi'(t)=\psi(t)$ for $t>0$; the latter equality implies that $\varphi'$ is nonzero on $(0,\infty)$.

For $f\in C^+(\sphere)$, let  \begin{equation}\label{functional-functional-11-28-21}
F(f)=\frac{1}{|\mu|}\int_{S^{n-1}}\varphi(f(u))\,d\mu(u),
\end{equation}
and for $K\in \cK_o^n$, define
$F(K)=F(h_K)$.  We claim that
\begin{equation}\label{optimization-convex-body-11-28-11}
\alpha=\inf\left\{F(K): \dveV(K)=|\mu| \ \mathrm{and}\ K\in \cK_o^n\right\}
\end{equation}
is well defined with $\alpha\in \R\cup\{-\infty\}$ because there is a $K \in \cK_o^n$ with $\dveV(K)=|\mu|$. To see this, note that
$$
\dveV(rB^n)=\int_{\sphere}G(r,u)\, du\ge \int_{\Sigma_\varepsilon(v)}G(r,u)\,du
$$
for any $v\in \sphere$ and $0<\varepsilon\le \varepsilon_0$. Then \eqref{condE2} yields $\dveV(rB^n)\to \infty$ as $r\to 0$, and $\dveV(rB^n)\to 0$ as $r\to \infty$. Since $r\to \dveV(rB^n)$ is continuous, there is an $r_0>0$ such that $\dveV(r_0B^n)=|\mu|$.  It follows from (\ref{optimization-convex-body-11-28-11}) that $\alpha\in \R\cup\{-\infty\}$.

By (\ref{optimization-convex-body-11-28-11}), there are $K_i \in\cK_o^n$, $i\in \N$, such that $\dveV(K_i)=|\mu|$ and
\begin{equation}\label{maximal-seq-11-28-11}
\lim_{i\rightarrow \infty} F(K_i)=\alpha.
\end{equation}
We aim to show that there is a $K_0 \in\cK_o^n$ with $\dveV(K_0)=|\mu|$ and $F(K_0)=\alpha$.

To this end, we first claim that there is an $R>0$ such that $K_i^{*} \subset R \ball$, $i\in \N$. Suppose on the contrary that $\sup_{i\in \N} R_{i}=\infty$, where $R_i=\max_{u\in\sphere}\rho_{K_{i}^{*}}(u)=\rho_{K_{i}^{*}}(v_i)$, say. By taking a subsequence, if necessary, we may suppose that $v_i\to v_0\in \sphere$ and $\lim_{i\to\infty}R_i=\infty$.  If $0<\ee\le \ee_0$ is given, there exists $i_{\varepsilon}\in \N$ such that  $|v_i-v_0|<\varepsilon/2$ whenever
$i\ge i_{\varepsilon}$.  Hence, if $u\in\Sigma_\varepsilon(v_0)$ and $i\ge i_{\varepsilon}$, then $\langle u,v_i\rangle\ge \varepsilon/2$.  It follows that for $u\in\Sigma_\varepsilon(v_0)$ and $i\ge i_{\varepsilon}$, we have
$$h_{K_i^{*}}(u)\ge \rho_{K_{i}^{*}}(v_i)\langle u, v_i\rangle=R_i\langle u, v_i\rangle\ge R_i\varepsilon/2$$
and therefore
\begin{eqnarray*}
|\mu|&=& \int_{\sphere} \wp(\rho_{K_i}(u), u)\,du
= \int_{\sphere} \wp(h_{K_i^*}(u)^{-1}, u)\,du\\
&\ge& \int_{\Sigma_\varepsilon(v_0)} \wp(h_{K_i^*}(u)^{-1}, u)\,du
\ge\int_{\Sigma_\varepsilon(v_0)} \wp( 2/(R_i\varepsilon), u)\,du \to \infty
\end{eqnarray*}
as $i\to\infty$.
This contradiction proves our claim.

By the Blaschke selection theorem, we may assume that $K_i^{*}\rightarrow L$ for some $L\in \cK^n$.  Suppose that $L\notin \cK_o^n$. Then $o\in \partial L$, so there exists $w_0\in S^{n-1}$ such that $\lim_{i\rightarrow \infty} h_{K_i^{*}}(w_0)=h_{L}(w_0)=0$.  Since $|\mu|>0$ and $\mu$ is not concentrated on any closed hemisphere, there is an $\ee\in (0,1)$ such that
$\mu(\Sigma_{\ee}(w_0))>0$. Let $v\in \Sigma_{\ee}(w_0)$. Since
$$
0\le \rho_{K_i^*}(v)\le \frac{1}{\langle v,w_0\rangle}h_{K_i^*}(w_0)\le \frac{1}{\varepsilon}h_{K_i^*}(w_0)\to 0
$$
as $i\to\infty$, it follows that $\rho_{K_i^{*}}\rightarrow 0$ uniformly on $\Sigma_{\ee} (w_0)$.  As $\dveV(K_i)=|\mu|$ and $K_i^{*} \subset R \ball$, using (\ref{bi-polar--1}), (\ref{functional-functional-11-28-21}), (\ref{optimization-convex-body-11-28-11}), and (\ref{maximal-seq-11-28-11}), we obtain
\begin{eqnarray*}
\alpha &=&\lim_{i\rightarrow \infty} F(K_i) =
\lim_{i\rightarrow \infty} \frac{1}{|\mu|}\int_{\sphere}\varphi \left(\rho_{K_i^*}(u)^{-1} \right)\,d\mu(u) \\
&\ge & \liminf_{i\rightarrow \infty} \frac{1}{|\mu|}\int_{\Sigma_{\varepsilon}(w_0)}\varphi \left(\rho_{K_i^*}(u)^{-1} \right)\,d\mu(u)+\frac{1}{|\mu|}\int_{\sphere\setminus \Sigma_{\varepsilon}(w_0)}\varphi \left(1/R \right)\,d\mu(u) \\
&\ge & \frac{\mu(\Sigma_\varepsilon(w_0))}{|\mu|}\liminf_{i\rightarrow \infty}\min\left\{\varphi\left(\rho_{K_i^*}(u)^{-1} \right):u\in\Sigma_{\varepsilon}(w_0)\right\}+\frac{\mu(\sphere\setminus \Sigma_{\varepsilon}(w_0))}{|\mu|}  \varphi \left(1/R \right)  =\infty.
\end{eqnarray*}
This is not possible, so $L\in \cK_o^n$.

Let $K_0=L^*\in \cK_o^n$. Then $K_i\to K_0$ as $i\to \infty$ in $\cK^n_o$. Hence, $h_{K_i} \rightarrow h_{K_0}>0$ uniformly on $\sphere$. The continuity of $\varphi$ ensures that
$$\sup\{|\varphi(h_{K_i}(u))|: i\in \N, u\in \sphere\}<\infty.$$
Now it follows from (\ref{functional-functional-11-28-21}), (\ref{maximal-seq-11-28-11}), and the dominated convergence theorem that
\begin{equation}\label{conv-11281159}
\alpha=\lim_{i\rightarrow \infty} F(K_i) = \frac{1}{|\mu|}\int_{S^{n-1}} \lim_{i\rightarrow \infty}\varphi(h_{K_i}(u))\,d\mu(u)=\frac{1}{|\mu|}
\int_{S^{n-1}}\varphi(h_{K_0}(u))\,d\mu(u)= F(K_0).
\end{equation}
Also, by Lemma~\ref{continuity-general-dual-qu-11-27}, we have $\dveV(K_0)=|\mu|$, so the aim stated earlier has been achieved.  It also follows from (\ref{conv-11281159}) that $\alpha\in \R$.

We now show that $K_0$ satisfies (\ref{msol}) with $K$ replaced by $K_0$.  Due to $\varphi\in {\mathcal{J}}_a$ and $f\geq h_{[f]}$, one has $F(f)\ge F(h_{[f]})= F([f])$ for $f\in C^+(\sphere)$.   By (\ref{conv-11281159}),
\begin{equation}\label{optimization-convex-body-11281157}
F(h_{K_0})=F(K_0)=\alpha=\inf\{F(f): \dveV([f])=|\mu| \ \mathrm{and}\ f\in C^+(\sphere)\}.
\end{equation}

Let $g\in C(\sphere)$.  For $u\in\sphere$ and sufficiently small $\ee_1,\ee_2\ge 0$, let $h_{\ee_1, \ee_2}$ be defined by (\ref{genplus}) with $f_0$ and $\ee g$ replaced by $h_{K_0}$ and $\ee_1 g +\ee_2$, respectively, i.e.,
\begin{equation}\label{feb172}
h_{\ee_1, \ee_2}(u) ={\varphi}^{-1}\left(\varphi(h_{K_0}(u))+\ee_1 g(u)+\ee_2\right).
\end{equation}
Then for sufficiently small $\ee$, we have
$$
h_{\ee_1+\ee,\ee_2}(u)=\varphi^{-1}\left(\varphi(h_{\ee_1,\ee_2}(u))+\ee g(u)\right)
$$
and
$$
h_{\ee_1,\ee_2+\ee}(u)=\varphi^{-1}\left(\varphi(h_{\ee_1,\ee_2}(u))+\ee\right).
$$
The properties of $\varphi$ listed after (\ref{feb222}) allow us to apply (\ref{variation-11-27-1}), with $\Omega=\sphere$ and with $h_0$ and $h_{\ee}$ replaced by $h_{\ee_1,\ee_2}$ and $h_{\ee_1+\ee,\ee_2}$, respectively, to obtain
\begin{equation}\label{variation-11-27-1---1}
\frac{\partial}{\partial\ee_1}\dveV([h_{\ee_1, \ee_2}])=
\lim_{\ee\rightarrow 0} \frac{\dveV([h_{\ee_1+\ee, \ee_2}])-\dveV([h_{\ee_1,\ee_2}])}{\ee}=
n\int_{\sphere} g(u)\,d\deV([h_{\ee_1,\ee_2}], u)
\end{equation}
and with $g$, $h_0$, and $h_{\ee}$ replaced by $1$, $h_{\ee_1,\ee_2}$ and $h_{\ee_1,\ee_2+\ee}$, respectively, to yield
\begin{equation}\label{variation-11-27-1---2}
\frac{\partial}{\partial\ee_2}\dveV([h_{\ee_1, \ee_2}]) =n\int_{\sphere} 1\,d\deV([h_{\ee_1,\ee_2}], u)=n\,\deV([h_{\ee_1,\ee_2}], \sphere)\neq 0.
\end{equation}
Since $[h_{\ee_1,\ee_2}]$ depends continuously on $\ee_1,\ee_2$ and in view of Proposition \ref{prop 11-28--1}(ii), (\ref{variation-11-27-1---1}) and (\ref{variation-11-27-1---2}) show that the gradient of the map $(\ee_1,\ee_2) \mapsto \dveV([h_{\ee_1, \ee_2}])$ has rank 1 and depends continuously on $(\ee_1,\ee_2)$, implying that this map is continuously differentiable.
Hence we may apply the method of Lagrange multipliers to conclude from (\ref{optimization-convex-body-11281157}) that there is a constant $\tau=\tau(g)$ such that
\begin{equation}\label{lagrange method}
\frac{\partial}{\partial \ee_1}\left(F(h_{\ee_1, \ee_2})+\tau(\log \dveV([h_{\ee_1, \ee_2}]) -\log |\mu|)\right)\Big|_{\ee_1=\ee_2=0}=0
\end{equation}
and
\begin{equation}\label{lagrange method--1}
\frac{\partial}{\partial \ee_2}\left(F(h_{\ee_1, \ee_2})+\tau(\log \dveV([h_{\ee_1, \ee_2}]) -\log |\mu|)\right)\Big|_{\ee_1=\ee_2=0}=0.
\end{equation}
By (\ref{functional-functional-11-28-21}) and (\ref{feb172}), we have
\begin{eqnarray}
\frac{\partial}{\partial \ee_1}F(h_{\ee_1, \ee_2})\Big|_{\ee_1=\ee_2=0}&=&\frac{1}{|\mu|}\left(\frac{\partial}{\partial \ee_1}\int_{\sphere}(\varphi(h_0(u))+\ee_1 g(u)+\ee_2)\,d\mu(u)\right)\Big|_{\ee_1=\ee_2=0} \nonumber \\
&=&\frac{1}{|\mu|}\int_{S^{n-1}} g(u)\,d\mu(u)\label{lag2}
\end{eqnarray}
and \begin{equation}\label{lag21}\ \ \ \
\frac{\partial}{\partial \ee_2}F(h_{\ee_1, \ee_2})\Big|_{\ee_1=\ee_2=0}=\frac{1}{|\mu|}\int_{S^{n-1}} 1\,d\mu(u)=1.
\end{equation}
Since $\dveV(K_0)=|\mu|$ and (\ref{feb172}) gives $h_{0,0}=h_{K_0}$, (\ref{variation-11-27-1---1}) and (\ref{variation-11-27-1---2}) imply that
\begin{equation}\label{lag1}
\frac{\partial}{\partial\ee_1}\log\dveV([h_{\ee_1, \ee_2}])\Big|_{\ee_1=\ee_2=0} =\frac{n}{|\mu|}\int_{\sphere} g(u)\,d\deV(K_0, u)
\end{equation}
and
\begin{equation}\label{lag1-2}
\frac{\partial}{\partial\ee_2}\log\dveV([h_{\ee_1, \ee_2}])\Big|_{\ee_1=\ee_2=0} =\frac{n}{|\mu|}\deV(K_0,\sphere).
\end{equation}
It follows from (\ref{lagrange method}), (\ref{lag2}), and (\ref{lag1}) that
\begin{equation}\label{identity-2018-0215-1}
\int_{S^{n-1}} g(u)\,d\mu(u) = -n\tau\int_{\sphere} g(u)\,d\deV(K_0, u)
\end{equation}
and from (\ref{lagrange method--1}), (\ref{lag21}), and (\ref{lag1-2}) that \begin{equation}\label{identity-2018-0215-2}
\tau = -\frac{|\mu|}{n\,\deV(K_0,\sphere)}.
\end{equation}
In particular, we see from (\ref{identity-2018-0215-2}) that $\tau$ is independent of $g$.  Finally, (\ref{identity-2018-0215-1}) and (\ref{identity-2018-0215-2}) show that (\ref{msol}) holds with $K$ replaced by $K_0$.
\end{proof}

We remark that $-\deV(K, \cdot)$ is a nonnegative measure since $G_t<0$.  Note that (\ref{condE2}) holds if $\lim_{t\rightarrow 0+}G(t, u)=\infty$ for $u\in \sphere$ and $\lim_{t\rightarrow \infty}G(t, u)=0$ for $u\in \Sigma_\ee (v)$. This follows from the monotone convergence theorem, since $t\mapsto G(t,u)$ is decreasing. In order to solve Problem~\ref{Minkowski-c-11-28} when $t\mapsto G(t,u)$ is increasing, one needs to use different techniques and we leave it for future work \cite{ghxy-1}.

When $\psi\equiv 1$ (and hence $\varphi(t)=\log t \in \mathcal{J}_{-\infty}$), the following result was proved in  \cite[Theorem~5.1]{XY2017-1}.
	
\bc \label{solution-general-dual-Orlicz-main theorem-11-26}
Let $\mu$ be a nonzero finite Borel measure on $S^{n-1}$ not concentrated on any closed hemisphere.  Let $\phi:\R^n\setminus\{o\}\to (0,\infty)$ be continuous and such that
$\overline{\Phi}$ is continuous on $(0,\infty)\times\sphere$, where $\overline{\Phi}$ is defined by \eqref{Phidef}. Suppose that for $v\in \sphere$ and $0<c<1$,
\begin{equation}\label{condE22}
\lim_{b\rightarrow 0+}\eV(C(v,b,c))=\infty,
\end{equation}
where $C(v,b,c)=\{x\in \R^n:|x|\ge b{\text{~and~}}\langle x/|x|, v\rangle\ge c\}$ and $\eV(\cdot)$ is defined by \eqref{general-quermass-11}. Let $\psi: (0,\infty)\rightarrow (0,\infty)$ be continuous and satisfy
\eqref{feb22}. Then there exists $K\in \cK_o^n$ such that
$$\frac{\mu}{|\mu|}=\frac{\ccV(K,\cdot)}{\ccV(K,S^{n-1})}.$$
\ec

\begin{proof}
By assumption,  $\of$ is continuous on $(0,\infty)\times\sphere$,  and $\lim_{t\rightarrow \infty} \of(t, u)=0$ for $u\in \sphere$. Hence the second condition in (\ref{condE2}) holds with $G$ replaced by $\of$.
Clearly, $\partial\of(t, u)/\partial t=-\phi(tu)t^{n-1}<0$.  By (\ref{condE22}),
$$\infty=\lim_{b\rightarrow 0+}\eV(C(v,b,c))=\lim_{b\rightarrow 0+}\int_{\Sigma_{c}(v)}\int_{b}^{\infty}\phi(ru)r^{n-1}drdu=
\lim_{b\rightarrow 0+}\int_{\Sigma_{c}(v)}\of(b, u)du.$$
Therefore the first condition in (\ref{condE2}) also holds with $G$ replaced by $\of$.  Since $\widetilde{C}_{\overline{\Phi},\psi}(K,\cdot)=-\ccV(K,\cdot)$,  Theorem~\ref{solution-general-dual-Orlicz-main theorem-11-27} yields the result.
\end{proof}

Another special case arises if $\mu$ is a discrete measure on $\sphere$, that is, $\mu=\sum_{i=1}^m c_i\delta_{v_i}$,
where $c_i>0$ for $i=1,\ldots,m$, and $v_1,\ldots,v_m\in \sphere$ are not contained in any closed hemisphere.  Let $G$ and $\psi$ be as in Theorem~\ref{solution-general-dual-Orlicz-main theorem-11-27}. Then there exists a polytope $P\in \cK_o^n$ such that
$$
\frac{\mu}{|\mu|}=\frac{\deV(P, \cdot)}{\deV(P, S^{n-1})}.
$$
To see this, note that Theorem~\ref{solution-general-dual-Orlicz-main theorem-11-27} ensures the existence of a $K\in \cK_o^n$ such that \eqref{msol} holds. Since $\mu$ is discrete, we obtain
$$\deV(K, \cdot)=\sum_{i=1}^m\bar c_i\delta_{v_i},$$
where $\bar c_i=\deV(K, S^{n-1})c_i/|\mu|<0$ for $i=1,\ldots,m$. Proposition~\ref{prop 11-28--1}(i) shows that there is a measurable function $g:\sphere\to (-\infty,0]$ such that
$$
\sum_{i=1}^m\bar c_i\delta_{v_i}(E)=\int_{E}g(u)\,dS(K,u)
$$
for Borel sets $E\subset \sphere$. Hence $S(K,\cdot)$ is a discrete measure and \cite[Theorem 4.5.4]{Sch} implies that $K$ is a polytope.

\section{Dual Orlicz-Brunn-Minkowski inequalities}\label{section7}
Let $\of_m$ be the set of continuous functions $\varphi: [0,\infty)^m\to [0,\infty)$ that are strictly increasing in each component and such that $\varphi(o)= 0$, $\varphi(e_j)=1$ for $1\leq j\leq m$, and $\lim_{t\to \infty} \varphi(tx) =\infty$ for $x\in [0,\infty)^m\setminus\{o\}$. By $\Psi_m$ we mean the set of continuous functions $\varphi: (0,\infty)^m\to(0,\infty)$, such that for $x=(x_1, \ldots, x_m) \in (0,\infty)^m$,
\begin{equation}\label{relation of G-12-06}
\varphi(x)=\varphi_0(1/x_1, \ldots, 1/x_m)
\end{equation}
for some $\varphi_0\in \of_m$. It is easy to see that
if $\varphi\in \Psi_m$, then $\varphi$ is strictly decreasing in each component and such that $\lim_{t\to 0} \varphi(tx) =\infty$ and $\lim_{t\to \infty} \varphi(tx) =0$ for $x\in (0,\infty)^m$.

Let $K_1, \ldots, K_m\in \cS^n_{c+}$ and let $\varphi\in\of_m\cup\Psi_m$. Define $\Oadd(K_1,\ldots,K_m)\in \cS_{c+}^n$, the {\em radial Orlicz sum} of $K_1, \ldots, K_m$, to be the star body whose radial function satisfies
\begin{equation}\label{Orldef}
\varphi\!\left( \frac{\rho_{K_1}(u)}{\rho_{\Oadd(K_1,\ldots,K_m)}(u)},\ldots, \frac{\rho_{K_m}(u)}{\rho_{\Oadd(K_1,\ldots,K_m)}(u)}\right) = 1
\end{equation}
for $u\in \sphere$.  It was proved in \cite[Theorem~3.2(v) and (vi)]{ghwy15} that if $\varphi\in \of_m$, then
\begin{equation}\label{Orldef-1206-1}
\rho_{\Oadd(K_1,\ldots,K_m)}(u)> \rho_{K_j}(u) \ \ \ \ \mbox{for}~ u\in \sphere.
\end{equation}
Together with (\ref{relation of G-12-06}) and (\ref{Orldef}), this implies that if $\varphi\in \Psi_m$, then
\begin{equation}\label{Orldef-1206-2} \rho_{\Oadd(K_1,\ldots,K_m)}(u)< \rho_{K_j}(u) \ \ \ \ \mbox{for}~ u\in \sphere.
\end{equation}
For each $0\neq q \in \R$ and $\varphi\in \of_m\cup\Psi_m$, let \begin{equation}\label{phiq}
\varphi_q(x)=\varphi\!\left(x_1^{1/q}, x_2^{1/q}, \ldots, x_m^{1/q}\right) \ \ \ \ \mbox{for~} x=(x_1, \ldots, x_m) \in (0, \infty)^m.
\end{equation}
Then (\ref{Orldef}) is equivalent to  \begin{equation}\label{Orldef-12-07}
\varphi_q\!\left( \left(\frac{\rho_{K_1}(u)}{\rho_{\Oadd(K_1,\ldots,K_m)}(u)}\right)^q,
\ldots,\left(\frac{\rho_{K_m}(u)}{\rho_{\Oadd(K_1,\ldots,K_m)}(u)}
\right)^q\right) = 1.
\end{equation}
For $t\in (0, \infty)$ and $u\in \sphere$, let
\begin{equation}\label{gdeff}
\wp_q(t, u)=\frac{\wp(t, u)}{t^q}.
\end{equation}

The proof of the following result closely follows that of \cite[Theorem~4.1]{ghwy15}.

\begin{theorem} \label{dual O-B-M-h}
Let $m, n\ge 2$, let $\varphi\in \of_m\cup\Psi_m$, let $K_1,\ldots, K_m\in \cS_{c+}^n$, let $\wp: (0, \infty)\times \sphere\rightarrow (0, \infty)$ be continuous, and let $\varphi_q$ and $\wp_q$ be defined by \eqref{phiq} and \eqref{gdeff}.  Suppose that $\varphi_q$ is convex and either $q>0$ and $\wp_q(t, \cdot)$ is increasing, or $q<0$ and $\wp_q(t, \cdot)$ is decreasing. Then
\begin{equation}\label{71}
1  \geq  \varphi\!\left(\left(\frac{\dveV(K_1) }{\dveV(\Oadd(K_1, \ldots, K_m))}\right)^{1/q},\ldots, \left(\frac{\dveV(K_m) }{\dveV(\Oadd(K_1, \ldots, K_m))}\right)^{1/q}\right). \end{equation}
The reverse inequality holds if instead $\varphi_q$ is concave and either $q>0$ and $\wp_q(t, \cdot)$ is decreasing, or $q<0$ and $\wp_q(t, \cdot)$ is increasing.

If in addition $\varphi_q$ is strictly convex (or convex, as appropriate) and equality holds in \eqref{71}, then $K_1,\dots, K_m$ are dilatates of each other.
\end{theorem}

\begin{proof}
Let $\varphi\in \of_m\cup \Psi_m$ and let $K_1, \ldots, K_m\in \cS_{c+}^n$. It follows from (\ref{Orldef}) that $\rho_{\Oadd(K_1, \ldots, K_m)}(u)>0$ for $u\in \sphere$. By (\ref{def-H-volume-12-07}), one can define a probability measure $\mu$ on $\sphere$  by
\begin{equation}\label{dualcone-12-07}
\,d\mu(u)= \frac{\wp(\rho_{\Oadd(K_1, \ldots, K_m)}(u), u)}{\dveV(\Oadd(K_1, \ldots, K_m))}\,du.
\end{equation}
Suppose that $\varphi\in \of_m$, $q>0$, and $\wp_q(t, \cdot)$ is increasing.  By (\ref{Orldef-12-07}) and Jensen's inequality \cite[Proposition 2.2]{ghwy15} applied to the convex function $\varphi_q$, similarly to the proof of \cite[Theorem~4.1]{ghwy15}, we have
\begin{eqnarray}
1&=&\int_{\sphere}\varphi_q\!\left( \left(\frac{\rho_{K_1}(u)}{\rho_{\Oadd(K_1,\ldots,K_m)}(u)}\right)^q,
\ldots, \left(\frac{\rho_{K_m}(u)}{\rho_{\Oadd(K_1,\ldots,K_m)}(u)}\right)^q
\right)\,d\mu(u) \nonumber \\
&\geq & \varphi_q\!\left(\int_{\sphere}\frac{\rho_{K_1}(u)^q}{
\rho_{\Oadd(K_1,\ldots, K_m)}(u)^q}\,d\mu(u),\ldots, \int_{\sphere}\frac{\rho_{K_m}(u)^q}{
\rho_{\Oadd(K_1,\ldots, K_m)}(u)^q}\,d\mu(u) \right)\label{d-O-B-M-12-07}.
\end{eqnarray}	
Since $\varphi\in \of_m$ and $q>0$, $\varphi_q$ is strictly increasing in each component. According to (\ref{Orldef-1206-1})  and the fact that $\wp_q(t, \cdot)$ is increasing, we have
\begin{equation}\label{feb23}
\frac{\rho_{K_j}(u)^q}{
\rho_{\Oadd(K_1,\ldots, K_m)}(u)^q} \wp(\rho_{\Oadd(K_1, \ldots, K_m)}(u), u) \geq \wp(\rho_{K_j}(u), u)
\end{equation}
for $j=1, \ldots, m$. Using (\ref{dualcone-12-07}), we obtain
\begin{eqnarray*}
\frac{\dveV(K_j) }{\dveV(\Oadd(K_1, \ldots, K_m))} &=& \frac{1}{\dveV(\Oadd(K_1, \ldots, K_m))} \int_{\sphere} \wp(\rho_{K_j}(u), u)\,du\\ &\leq &\frac{1}{\dveV(\Oadd(K_1, \ldots, K_m))}  \int_{\sphere}\frac{\rho_{K_j}(u)^q\, \wp(\rho_{\Oadd(K_1, \ldots, K_m)}(u), u)  }{
\rho_{\Oadd(K_1,\ldots, K_m)}(u)^q} \,du \\&=&  \int_{\sphere}\frac{\rho_{K_j}(u)^q}{
\rho_{\Oadd(K_1,\ldots, K_m)}(u)^q} \,d\mu(u)
\end{eqnarray*}
for $j=1, \ldots, m$.  Since $\varphi_q$ is strictly increasing in each component and (\ref{d-O-B-M-12-07}) holds, we get  \begin{eqnarray}
1 &\geq & \varphi_q\left(\int_{\sphere}\frac{\rho_{K_1}(u)^q}{
\rho_{\Oadd(K_1,\ldots, K_m)}(u)^q}\,d\mu(u),\ldots, \int_{\sphere}\frac{\rho_{K_m}(u)^q}{
\rho_{\Oadd(K_1,\ldots, K_m)}(u)^q}\,d\mu(u) \right)\nonumber \\&\geq & \varphi_q\left(\frac{\dveV(K_1) }{\dveV(\Oadd(K_1, \ldots, K_m))},\ldots, \frac{\dveV(K_m) }{\dveV(\Oadd(K_1, \ldots, K_m))} \right)\nonumber\\&=&\varphi \left(\left(\frac{\dveV(K_1) }{\dveV(\Oadd(K_1, \ldots, K_m))}\right)^{1/q},\ldots, \left(\frac{\dveV(K_m) }{\dveV(\Oadd(K_1, \ldots, K_m))}\right)^{1/q}\right),\label{feb3}
\end{eqnarray}
which yields (\ref{71}).

Suppose in addition that $\varphi_q$ is strictly convex and equality holds in (\ref{71}).  Then equality holds throughout (\ref{feb3}) and hence in (\ref{d-O-B-M-12-07}).  Therefore equality holds in Jensen's inequality as used above.  Since $G>0$, the definition (\ref{dualcone-12-07}) of $\mu$ shows that its support is the whole of $S^{n-1}$. Then, exactly as in the proof of \cite[Theorem~4.1]{ghwy15}, we can conclude that $K_1,\dots,K_m$ are dilatates of each other.

This proves (\ref{71}) and the implication in case of equality when $\varphi\in \of_m$, $q>0$, and $\wp_q(t, \cdot)$ is increasing.  The other cases are similar, noting that if $\varphi\in \Psi_m$, we can use (\ref{Orldef-1206-2}) instead of (\ref{Orldef-1206-1}), and if $\varphi_q$ is concave, Jensen's inequality \cite[Proposition 2.2]{ghwy15} yields the reverse of inequality (\ref{d-O-B-M-12-07}).
\end{proof}

It is possible to state more general versions of Theorem~\ref{dual O-B-M-h} that hold when $K_1,\ldots, K_m\in \cS^n$. Indeed, the definition (\ref{Orldef}) of the radial Orlicz sum can be modified, as in \cite[p.~817]{ghwy15}, so that it applies when $K_1,\ldots, K_m\in \cS^n$.  Then extra assumptions would have to be made in Theorem~\ref{dual O-B-M-h}, analogous to the one in \cite[Theorem~4.1]{ghwy15} that $V_n(K_j)>0$ for some $j$, but now also involving the function $G$ .  Note that the stronger assumption that $K_1,\ldots, K_m\in \cS_{c+}^n$ is still required for the implication in case of equality, as it is in \cite[Theorem~4.1]{ghwy15}.

Under certain circumstances, equality holds in Theorem~\ref{dual O-B-M-h} if and only if $K_1,\dots, K_m$ are dilatates of each other.  One such is given in Corollary~\ref{dual O-B-M-h-1}, and it is easy to see that this is true more generally if $G$ is of the form $G(t,u)=t^qH(u)$, where $t>0$ and $u\in\sphere$, for some $q\neq 0$ and suitable function $H$, since equality then holds in (\ref{feb23}). However, it does not seem straightforward to formulate a precise condition and we do not pursue the matter here.

Dual Orlicz-Brunn-Minkowski inequalities for $\eV(\cdot)$, $\teV(\cdot)$ and $\oveV(\cdot, \cdot)$ follow directly from Theorem~\ref{dual O-B-M-h}, once the corresponding assumptions are verified.  We shall only state the special case when $\wp(t, u)=t^q\rho_Q(u)^{n-q}/n$ for some $Q\in \cS_{c+}^n$. Then, for $q\neq 0$, we have
\begin{equation}\label{q-mixed-12-09}
\dveV(K)=\int_{\sphere} \wp(\rho_K(u), u)\,du =\frac{1}{n}\int_{\sphere}\rho_K(u)^q\,\rho_Q(u)^{n-q}\,du=\teVq(K, Q),
\end{equation}
the $q$th dual mixed volume of $K$ and $Q$, as in (\ref{qmixedv}).

The following result was proved for $q=n$ and $Q=\ball$ in \cite[Theorem~4.1]{ghwy15}.

\bc \label{dual O-B-M-h-1}
Let $m, n\ge 2$, let $q\neq 0$, let $\varphi\in \of_m\cup\Psi_m$, and let $Q, K_1, \ldots, K_m\in \cS_{c+}^n$. If $\varphi_q$ is convex, then
\begin{equation}\label{feb224}
1  \geq  \varphi\!\left(\left(\frac{\teVq(K_1, Q)}{\teVq(\Oadd(K_1, \ldots, K_m), Q)}\right)^{1/q},\ldots, \left(\frac{\teVq(K_m, Q) }{\teVq(\Oadd(K_1, \ldots, K_m), Q)}\right)^{1/q}\right).
\end{equation}
If $\varphi_q$ is concave, the inequality is reversed. If instead $\varphi_q$ is strictly convex or strictly concave, respectively, then
equality holds in \eqref{71} if and only if $K_1, \ldots, K_m$ are dilatates of each other.
\ec

\begin{proof}
The required inequalities and the necessity of the equality condition follow immediately from Theorem~\ref{dual O-B-M-h} on noting that $\wp_q(t, u)=\rho_Q(u)^{n-q}/n$ is a constant function of $t$.

Suppose that $K_1, \ldots, K_m$ are dilatates of each other, so $K_i=c_iK$ and hence $\rho_{K_i}=c_i\rho_K$ for some $K\in \cS_{c+}^n$ and $c_i>0$, $i=1,\dots,m$.  Let $d>0$ be the unique solution of
\begin{equation}\label{feb223}
\varphi\left(\frac{c_1}{d},\dots,\frac{c_m}{d}\right)=1.
\end{equation}
Comparing (\ref{Orldef}), we obtain $\rho_{\Oadd(K_1,\ldots,K_m)}(u)=d\rho_K(u)$ for $u\in\sphere$ and hence we have $\Oadd(K_1,\ldots,K_m)=dK$.  From (\ref{q-mixed-12-09}), we get $\teVq(K_i, Q)=c_i^q\,\teVq(K, Q)$, $i=1,\dots,m$, and $\teVq(\Oadd(K_1,\ldots,K_m), Q)=d^q\,\teVq(K, Q)$.  Substituting for $c_i$, $i=1,\dots,m$, and $d$ from the latter two equations into (\ref{feb223}), we obtain (\ref{feb224}) with equality.
\end{proof}

\section{Dual Orlicz-Minkowski inequalities and uniqueness results}\label{section8}
Let $K, L, Q\in \cS_{c+}^n$, let $q\neq 0$, and let $\varphi: (0, \infty)\rightarrow (0, \infty)$ be continuous.  It will be convenient to define
\begin{equation}\label{H-def-mixed-12-09-1}
\HeVq(K,L,Q)=\frac{1}{n}\int_{\sphere} \varphi\left(\frac{\rho_{L}(u)}{\rho_{K}(u)}\right)\,\rho_{K}(u)^q \, \rho_Q(u)^{n-q} \,du.
\end{equation}
Note that this is a special case of the general dual Orlicz mixed volume $\veV(K, L)$ defined in (\ref{dua1}), obtained by setting $\phi(x)=|x|^{q-n}\rho_Q(x/|x|)^{n-q}$.  When $q=n$, (\ref{H-def-mixed-12-09-1}) becomes the dual Orlicz mixed volume introduced in \cite{ghwy15, Zhub2014}, and when $q=n$ and $Q=\ball$, the following result yields the dual Orlicz-Minkowski inequality established in \cite[Theorem~6.1]{ghwy15} and \cite[Theorem~5.1]{Zhub2014}.

\begin{theorem} \label{dual-O-M-I-1209}
Let $K, L, Q\in \cS_{c+}^n$, let $q\neq 0$, let $\varphi: (0, \infty)\rightarrow (0, \infty)$ be continuous, and let $\varphi_q(t)=\varphi(t^{1/q})$ for $t\in (0, \infty)$.  If $\varphi_q$ is convex, then
\begin{eqnarray}\label{81}
\HeVq(K, L, Q)\ge \teVq(K, Q)\,\varphi\!\left(\left(\frac{\teVq(L, Q) }{\teVq(K, Q)}\right)^{1/q}\right).
\end{eqnarray}
The reverse inequality holds if $\varphi_q$ is concave.
If $\varphi_q$ is strictly convex or strictly concave, respectively, equality holds in the above inequalities if and only if $K$ and $L$ are dilatates of each other.
\end{theorem}

\begin{proof}
Let $q\neq 0$ and let $\varphi_q$ be convex. By (\ref{q-mixed-12-09}), one can define a probability measure $\tilde{\mu}$ by
$$
d\tilde{\mu}(u)=\frac{\rho_{K}(u)^q\,\rho_Q(u)^{n-q}}{n\teVq(K, Q)}\,du.
$$
 Jensen's inequality \cite[Proposition 2.2]{ghwy15} implies that
\begin{eqnarray*}
\HeVq(K, L, Q)&=&\frac{1}{n}\int_{\sphere} \varphi\left(\frac{\rho_{L}(u)}{\rho_{K}(u)}\right)\,\rho_{K}(u)^q \, \rho_Q(u)^{n-q} \,du \\ &=&\teVq(K, Q) \int_{\sphere} \varphi_q\!\left(\left(\frac{\rho_{L}(u)}{\rho_{K}(u)}\right)^q\right)  \,d\tilde{\mu}(u) \\&\geq& \teVq(K, Q)\,\varphi_q\!\left( \int_{\sphere}\left(\frac{\rho_{L}(u)}{\rho_{K}(u)}\right)^q  \,d\tilde{\mu}(u)\right) \\
&=& \teVq(K, Q)\,\varphi_q\!\left(\int_{\sphere}\frac{ \rho_{L}(u)^q \,\rho_Q(u)^{n-q}}{n\teVq(K, Q)} \,du\right) \\
&=& \teVq(K, Q) \,\varphi\!\left(\left(\frac{\teVq(L, Q) }{\teVq(K, Q) }\right)^{1/q}\right),
\end{eqnarray*}
where the first and the last equalities are due to (\ref{H-def-mixed-12-09-1}) and (\ref{q-mixed-12-09}), respectively.

Suppose that $\varphi_q$ is strictly convex and equality holds in (\ref{81}).  Then the above proof and the equality condition for Jensen's equality show that $\rho_{L}(u)/\rho_{K}(u)$ is a constant for $\tilde{\mu}$-almost all $u\in \sphere$ and hence for ${\mathcal{H}}^{n-1}$-almost all $u\in \sphere$. Since $\rho_K$ and $\rho_L$ are continuous, $\rho_{L}(u)/\rho_{K}(u)$ is a constant for $u\in \sphere$ and so $K$ and $L$ are dilatates of each other.

If instead $\varphi_q$ is concave, the proof is similar since Jensen's inequality \cite[Proposition 2.2]{ghwy15} also reverses. \end{proof}

\bc
Let $K, L, Q\in \cS_{c+}^n$, let $q\neq 0$, let $\varphi:(0,\infty)\to (0,\infty)$, and let $\varphi_q(t)=\varphi(t^{1/q})$ for $t\in (0, \infty)$.  Suppose that $\varphi$ is either increasing or decreasing, and that $\varphi_q$ is either strictly convex or strictly concave.  Then $K=L$ if either
\begin{equation}\label{formula-1218-1}
\frac{\HeVq(K, M, Q)}{\teVq(K, Q) }=\frac{\HeVq(L, M, Q)}{\teVq(L, Q)}
\end{equation}
holds for all $M \in {\mathcal{S}}_{c+}^n$, or
\begin{equation}\label{formula-1218-2}
\HeVq(M, K, Q)=\HeVq(M, L, Q)
\end{equation}
holds for all $M \in {\mathcal{S}}_{c+}^n$.
\ec

\begin{proof}
Let $q\neq 0$ and suppose that (\ref{formula-1218-1}) holds for all $M\in \cS_{c+}^n$.  Assume that $\varphi$ is increasing and $\varphi_q$ is strictly convex; the other three cases can be dealt with similarly.  Taking $M=K$ in (\ref{formula-1218-1}), it follows from (\ref{qmixedv}), (\ref{H-def-mixed-12-09-1}) with $L=K$, and (\ref{81}) with $K$ and $L$ interchanged, that
\begin{equation}\label{appl-12-10}
\varphi(1)= \frac{\HeVq(K, K, Q)}{\teVq(K, Q) } = \frac{\HeVq(L, K, Q)}{\teVq(L, Q) }\ge \varphi\!\left(\left(\frac{\teVq(K, Q) }{\teVq(L, Q) }\right)^{1/q}\right).
\end{equation}
Since $\varphi$ is increasing, we get
\begin{equation}\label{appl-ineq-11}
1 \geq \left(\frac{\teVq(K, Q) }{\teVq(L, Q) }\right)^{1/q}.
\end{equation}
Repeating the argument with $K$ and $L$ interchanged yields the reverse inequality. Hence we get $\teVq(K, Q)=\teVq(L, Q)$, from which we obtain equality in (\ref{appl-12-10}). The equality condition for (\ref{81}) implies that $L=rK$ for some $r>0$.  This together with $\teVq(K, Q)=\teVq(L, Q)$ easily yields $K=L$.

Now suppose that (\ref{formula-1218-2}) holds for all $M\in \cS_{c+}^n$.  Taking $M=K$ and arguing as above, we get
\begin{equation}\label{appl-ineq-222}
\varphi(1)\,\teVq(K, Q) = \HeVq(K, K, Q) = \HeVq(K, L, Q) \ge  \teVq(K, Q)\, \varphi\!\left(\left(\frac{\teVq(L, Q) }{\teVq(K, Q) }\right)^{1/q}\right).
\end{equation}
Therefore (\ref{appl-ineq-11}) holds.  Interchanging $K$ and $L$ yields the reverse inequality and hence we have $\teVq(K, Q)=\teVq(L, Q)$, giving equality in (\ref{appl-ineq-222}). Exactly as above, we conclude that $K=L$.
\end{proof}

\bc\label{thm1s}
Let $K, L, Q\in \cS_{c+}^n$, let $q\neq 0$, let $\varphi: (0, \infty)\rightarrow (0, \infty)$ be continuous, and let $\varphi_q(t)=\varphi(t^{1/q})$ for $t\in (0, \infty)$.  If $\varphi_q$ is strictly convex or strictly concave and
\begin{equation}\label{83}
\HeVq(K, M, Q)=\HeVq(L, M, Q)
\end{equation}
for all $M\in {\mathcal{S}}_{c+}^n$, then $K=L$.
\ec

\begin{proof}
Let $q\neq 0$ and let $\alpha>0$.  Replacing $K$ and $L$ by $L$ and $\alpha L$, respectively, in (\ref{H-def-mixed-12-09-1}), and taking (\ref{q-mixed-12-09}) into account, we obtain,
$$ \HeVq(L, \alpha L, Q)=\frac{\varphi(\alpha)}{
\varphi(1)}\HeVq(L, L, Q)=\varphi(\alpha)\teVq(L, Q).$$

Suppose that $\varphi_q$ is strictly convex; the case when $\varphi_q$ is strictly concave is similar. Using (\ref{83}) with $M=\alpha L$, (\ref{81}) implies that
\begin{equation}\label{oeq1-12-10}
\varphi(\alpha)\teVq(L, Q) =\HeVq(L, \alpha L, Q) =\HeVq(K, \alpha L, Q) \ge \teVq(K, Q) \, \varphi\!\left(\alpha \left(\frac{\teVq(L, Q) }{\teVq(K, Q) }\right)^{1/q}\right).
\end{equation}
Let $$c= \left(\frac{\teVq(L, Q) }{\teVq(K, Q)}\right)^{1/q}.$$  Then (\ref{oeq1-12-10}) reads $c^q\varphi(\alpha) \ge \varphi(\alpha c)$.  When $\alpha=1$, we obtain
\begin{equation}\label{oeq2}
c^q\varphi(1)\ge \varphi(c).
\end{equation}
Repeating the argument with $K$ and $L$ interchanged yields $c^{-q}\varphi(\alpha)\ge \varphi(\alpha c^{-1})$.  Setting $\alpha=c$, we get $c^{-q}\varphi(c)\ge \varphi(1)$ and hence
\begin{equation}\label{oeq3}
c^q\varphi(1)\le \varphi(c).
\end{equation}
By (\ref{oeq2}) and (\ref{oeq3}), $\varphi(c)=c^q\varphi(1)$, which means that
$$\varphi\!\left(\left(\frac{\teVq(L, Q) }{\teVq(K, Q)}\right)^{1/q}\right)=\frac{\teVq(L, Q) }{\teVq(K, Q)}\,\varphi(1).$$
Thus equality holds in (\ref{oeq1-12-10}) when $\alpha=1$. By the equality condition for (\ref{81}),  we conclude that $L=rK$ for some $r>0$. That is, $K$ and $L$ are dilatates of each other.

Suppose that $L=rK$, where $r>0$ and $r\neq 1$. Let $\alpha>0$.  Then (\ref{q-mixed-12-09}), (\ref{H-def-mixed-12-09-1}), and (\ref{83}) with $M=\alpha K$ yield
$$\varphi(\alpha)\teVq(K, Q)=\HeVq (K, \alpha K, Q)=\HeVq (rK, \alpha K, Q)=\varphi(\alpha/r)r^q \,\teVq(K, Q).$$
Consequently, $\varphi(rs)=r^q\varphi(s)$ for $s>0$. Equivalently, setting $\beta=r^q$ and $t=s^q$, we obtain $\varphi_q(\beta t)=\beta\varphi_q(t)$ for $t>0$, where $\beta\neq 1$.  But then the points $(\beta^m, \varphi_q(\beta^m))$, $m\in \N$, all lie on the line $y=\varphi(1)x$ in $\R^2$, so $\varphi_q$ cannot be strictly convex.  This contradiction proves that $r=1$ and hence $K=L$.
\end{proof}

Let $K, L\in {\mathcal{K}}_{o}^n$. We recall from \cite{GHW2014, XJL} that for $\varphi\in(0, \infty)\rightarrow (0, \infty)$, the {\em Orlicz mixed volume} $V_{\varphi}(K,L)$ is defined by
\begin{equation}\label{intreprLp}
V_{\varphi}(K,L)=\frac{1}{n}\int_{S^{n-1}}
\varphi\left(\frac{h_L(u)}{h_K(u)}\right)h_K(u)\,dS(K,u).
\end{equation}
The {\em Orlicz-Minkowski inequality} \cite[Theorem~9.2]{GHW2014} (see also \cite[Theorem~2]{XJL}) states that if $\varphi\in \mathcal{I}$ is convex, then
\begin{equation}\label{NewMI}
V_{\varphi}(K,L)\ge V_n(K)\,\varphi\!\left(\left(\frac{V_n(L)}{V_n(K)}\right)^{1/n}\right).
\end{equation}
If $\varphi$ is strictly convex, equality in (\ref{NewMI}) holds if and only if $K$ and $L$ are dilatates of each other.  When $\varphi(t)=t$, we write $V_{\varphi}(K,L)=V_1(K,L)$ and retrieve from (\ref{NewMI}) {\em Minkowski's first inequality}
\begin{equation}\label{M122}
V_1(K,L)\ge V_n(K)^{(n-1)/n}V_n(L)^{1/n}.
\end{equation}
Note that (\ref{M122}) actually holds for all $K,L\in \cK^n$, with equality if and only if $K$ and $L$ lie in parallel hyperplanes or are homothetic; see
\cite[Theorem~B.2.1]{Gar06} or \cite[Theorem~6.2.1]{Sch}.

Let $\varphi\in \mathcal{I}\cup\mathcal{D}$ and let $n\in \N$, $n\ge 2$.  We say that $\varphi$ {\em behaves like $t^n$} if there is $r>0$, $r\neq 1$, such that $\varphi(rt)=r^n\varphi(t)$ for $t>0$.  Of course, if $\varphi(t)=t^n$, then $\varphi$ behaves like $t^n$, but there is a $\varphi\in \mathcal{I}\cup\mathcal{D}$ that behaves like $t^n$ such that $\varphi(t)\neq t^n$ for some $t>0$. To see this, let $f(t)=t^n$ and define $\varphi(t)$ on $[1,2]$, such that (i) $\varphi$ is increasing and strictly convex, (ii) $\varphi(t)=f(t)$ at $t=1$ and $t=2$, (iii) $\varphi'_r(1)=f'(1)$ and $\varphi'_l(2)=f'(2)$, and (iv) $\varphi(t)<f(t)$ on $(1,2)$.  Then define $\varphi$ on $[1/2,1]$ by $\varphi(t)=\varphi(2t)/2^n$ and on $[2,4]$ by $\varphi(t)=2^n\varphi(t/2)$. It follows that $\varphi$ is increasing and strictly convex on $[1/2,1]$ and on $[2,4]$, $\varphi(t)=f(t)$ at $t=1/2$ and $t=4$,  $\varphi'_r(1/2)=\varphi'_r(1)/2^{n-1}=f'(1/2)$, $\varphi'_l(4)=2^{n-1}\varphi_l'(2)=f'(4)$ and $\varphi(t)<f(t)$ on $(1/2,1)\cup (2,4)$.  Moreover, $\varphi'_l(t)=\varphi'_r(t)$ at $t=1$ and $t=2$, so $\varphi$ is increasing and strictly convex on $[1/2,4]$.  Continuing inductively, we define $\varphi$ on $[1/2^m,2^{m+1}]$, $m\in \N$, and hence on $(0,\infty)$, so that it is increasing and strictly convex, $\varphi(t)=t^n$ for $t=1/2^m$ and $t=2^m$, $m\in \N$, and $\varphi(t/2)=2^{-n}\varphi(t)$ for $t>0$,
but $\varphi$ is not identically equal to $t^n$.  This construction for $r=1/2$ (or, equivalently, $r=2$) can be easily modified for other values of $r>0$, $r\neq 1$.

The following result can be obtained from (\ref{NewMI}) and the argument in the proof of Corollary~\ref{thm1s}.

\bc\label{thm1}
Let $K, L\in {\mathcal{K}}_o^n$. Suppose that $\varphi\in \mathcal{I}$ is strictly convex and $V_{\varphi}(K, M)=V_{\varphi}(L, M)$ for all $M\in {\cK}^n_{o}$.  Then $K$ and $L$ are dilatates of each other.  Moreover, $K=L$ unless $\varphi$ behaves like $t^n$.
\ec

Note that the restriction in the second statement of the previous theorem is necessary, since it is evident from (\ref{intreprLp}) that if $\varphi$ behaves like $t^n$, then for the corresponding $r\neq 1$, we have $V(K, M)=V(rK, M)$ for all $M\in {\cK}^n_{o}$.

Let $K, L\in {\mathcal{K}}_o^n$, let $Q\in \cS^n_{c+}$, and let $p,q\in \R$.  In \cite[(1.13), p.~91]{LYZ-Lp}, the {\em $(p,q)$-mixed volume} $\widetilde{V}_{p,q}(K,L,Q)$ was defined by setting $g=h_L^p$ in (\ref{LYZint}):
\begin{eqnarray}\label{LYZmv}
\widetilde{V}_{p,q}(K,L,Q)&=&\int_{\sphere} h_L(u)^p\, d\widetilde{C}_{p,q}(K,Q,u)\nonumber\\
&=&\frac{1}{n}\int_{\sphere}h_L(\alpha_K(u))^p\,
h_{K}(\alpha_K(u))^{-p}\,\rho_{K}(u)^q\,\rho_Q(u)^{n-q}\,du.
\nonumber\\
&=&\frac{1}{n}\int_{\sphere}\left(\frac{h_L(\alpha_K(u))}
{h_{K}(\alpha_K(u))}\right)^p\,
\left(\frac{\rho_{K}(u)}{\rho_Q(u)}\right)^q\,\rho_Q(u)^{n}\,du.
\end{eqnarray}
Inspired by (\ref{LYZmv}), we can consider the nonlinear Orlicz dual curvature functionals defined by
$$
\frac{1}{n}\int_{S^{n-1}}
\varphi\!\left(\psi\!\left(\frac{f(\alpha_K(u))}{h_K(\alpha_K(u))}\right)
\left(\frac{\rho_K(u)}{\rho_Q(u)}\right)^n\right)\,\rho_Q(u)^n\,du,
$$
where $\varphi, \psi: (0, \infty)\rightarrow (0, \infty)$ are continuous functions and $f\in C^+(S^{n-1})$.  We can then take $f=h_L$ to define the {\em $(\varphi,\psi)$-mixed volume}
$$
\widetilde{V}_{\varphi,\psi}(K,L,Q)=\frac{1}{n}\int_{S^{n-1}}
\varphi\!\left(\psi\!\left(\frac{h_L(\alpha_K(u))}{h_K(\alpha_K(u))}\right)
\left(\frac{\rho_K(u)}{\rho_Q(u)}\right)^n\right)\,\rho_Q(u)^n\,du.
$$
This is a natural generalization of (\ref{LYZmv}) when $q\neq 0$, corresponding to taking $\varphi(t)=t^{q/n}$ and $\psi(t)=t^{np/q}$.

When $L\in {\mathcal{K}}_{o}^n$, the following result provides a common generalization of \cite[Theorem~9.2]{GHW2014}, \cite[Theorem~6.1]{ghwy15} (see also \cite[Theorem~2]{Zhub2014}), and \cite[Theorem~7.4]{LYZ-Lp}.  The first corresponds to taking $K=Q$ when $\varphi$ there is replaced by $\varphi\circ\psi$, the second corresponds to taking $K=L$, and the third is obtained by the choices of $\varphi$ and $\psi$ given in the previous paragraph.  Note that in the latter case, for the convexity of $\varphi$ and $\psi$ we then require that $1\le q/n\le p$, which is precisely the assumption made in \cite{LYZ-Lp}.

\bt\label{thm2}
Let $K,L\in {\mathcal{K}}_{o}^n$ and let $Q\in {\mathcal{S}}_{c+}^n$. If $\varphi, \psi: (0, \infty)\rightarrow (0, \infty)$ are increasing and convex, then
\begin{equation}\label{Orliczgmvineq}
\widetilde{V}_{\varphi,\psi}(K,L,Q)\ge
\varphi\!\left(\frac{V_n(K)}{V_n(Q)}\,\psi\!
\left(\left(\frac{V_n(L)}{V_n(K)}\right)^{1/n}\right)\right)V_n(Q).
\end{equation}
If $\varphi$ and $\psi$ are strictly convex, equality holds if and only if $K$, $L$, and $Q$ are dilatates of each other.
\et

\begin{proof}
Setting $Q=K$ and $p=1$ in \cite[(7.6), Proposition~7.2]{LYZ-Lp}, (\ref{LYZmv}), and (\ref{LYZint}), we have, for any $q\neq 0$,
\begin{eqnarray}\label{feb16}
V_1(K,L)&=& \widetilde{V}_{1,q}(K,L,K)\nonumber\\
&=&\int_{\sphere}h_L(u)\,d{\widetilde{C}}_{1,q}(K,K,u)\nonumber\\
&=&\frac{1}{n}\int_{\sphere}\frac{h_L(\alpha_K(u))}{h_K(\alpha_K(u))}
\rho_K(u)^n\,du.
\end{eqnarray}
We use Jensen's inequality \cite[Proposition 2.2]{ghwy15} twice, once with $\varphi$ and once with $\psi$, Minkowski's first inequality (\ref{M122}), and (\ref{feb16}) to obtain
\begin{eqnarray*}
\frac{\widetilde{V}_{\varphi,\psi}(K,L,Q)}{V_n(Q)}
&=&\frac{1}{n V_n(Q)}\int_{S^{n-1}}
\varphi\!\left(\psi\!\left(\frac{h_L(\alpha_K(u))}{h_K(\alpha_K(u))}\right)
\left(\frac{\rho_K(u)}{\rho_Q(u)}\right)^n\right)\rho_Q(u)^n\,du\\
&\ge &\varphi\!\left(\frac{1}{n V_n(Q)}\int_{S^{n-1}}\psi\!\left(\frac{h_L(\alpha_K(u))}{h_K(\alpha_K(u))}\right)
\left(\frac{\rho_K(u)}{\rho_Q(u)}\right)^n\rho_Q(u)^n\,du\right)\\
&=&\varphi\!\left(\frac{V_n(K)}{V_n(Q)}\cdot \frac{1}{n V_n(K)}\int_{S^{n-1}}\psi\!\left(\frac{h_L(\alpha_K(u))}{h_K(\alpha_K(u))}\right)
\rho_K(u)^n\,du\right)\\
&\ge &\varphi\!\left(\frac{V_n(K)}{V_n(Q)}\,\psi\!\left(\frac{1}{n V_n(K)}\int_{S^{n-1}}\frac{h_L(\alpha_K(u))}{h_K(\alpha_K(u))}
\rho_K(u)^n\,du\right)\right)\\
&= &\varphi\!\left(\frac{V_n(K)}{V_n(Q)}\,\psi\!\left(\frac{V_1(K,L)}{ V_n(K)}\right)\right)\\
&\ge &
\varphi\!\left(\frac{V_n(K)}{V_n(Q)}\,\psi\!
\left(\left(\frac{V_n(L)}{V_n(K)}\right)^{1/n}\right)\right),
\end{eqnarray*}
as required.

Suppose that $\varphi$ and $\psi$ are strictly convex and that equality holds in (\ref{Orliczgmvineq}).  Then equality holds throughout the previous display.  As in the proof of \cite[Lemma~9.1]{GHW2014}, equalities in Minkowski's first inequality and in Jensen's inequality with $\psi$ implies that $K$ and $L$ are dilatates of each other.  Then equality in Jensen's inequality with $\varphi$ implies that $K$ and $Q$ are dilatates of each other.
\end{proof}

We omit the proof of the following corollary, which is again similar to that of Corollary~\ref{thm1s}.

\bc\label{onemorecor}
Let $K,L\in {\mathcal{K}}_{o}^n$, and suppose that $\varphi, \psi: (0, \infty)\rightarrow (0, \infty)$ are increasing and strictly convex. If $\widetilde{V}_{\varphi,\psi}(K,M,Q)=\widetilde{V}_{\varphi,\psi}(L,M,Q)$ for $M=\alpha K$, $\alpha>0$, $Q=K$ and for $M=\alpha L$, $\alpha>0$, $Q=L$, then $K$ and $L$ are dilatates of each other. Moreover, $K=L$ unless $\psi$ behaves like $t^n$. If $\psi$ behaves like $t^n$ with $\psi(rt)=r^n\psi(t)$, $t>0$, for some $r>0$, then
$\widetilde{V}_{\varphi,\psi}(K,M,Q)=\widetilde{V}_{\varphi,\psi}(rK,M,Q)$ for all $K,M,Q\in\cK^n_o$.
\ec

\end{document}